\documentclass[12pt,reqno]{amsart}
\usepackage[margin=1in]{geometry}
\usepackage{amsmath,amssymb,amsthm,graphicx,amsxtra, setspace}
\usepackage[utf8]{inputenc}
\usepackage{mathrsfs}
\usepackage{hyperref}
\usepackage{xcolor}
\usepackage{upgreek}
\usepackage{mathtools,verbatim}
\allowdisplaybreaks
\newtheorem{theorem}{Theorem}[section]
\newtheorem{lem}[theorem]{Lemma}
\newtheorem{rem}[theorem]{Remark}
\newtheorem{cor}[theorem]{Corollary}
\newtheorem{Def}[theorem]{Definition}
\newtheorem{pro}[theorem]{Proposition}
\newtheorem{Ex}[theorem]{Example}

\usepackage{latexsym}
\usepackage{hyperref}
\usepackage{graphicx}
\usepackage{epstopdf}

\DeclareMathOperator*{\esssup}{ess\,sup}

\allowdisplaybreaks

\let\originalleft\left
\let\originalright\right
\renewcommand{\left}{\mathopen{}\mathclose\bgroup\originalleft}
\renewcommand{\right}{\aftergroup\egroup\originalright}

\newcommand{\Addresses}{{
		\footnote{
			\footnotesize
			
			\noindent  Sumit  Arora\\ \noindent  \texttt{arorasumit10623@gmail.com, sumit.arora@utalca.cl}\\
   
			\noindent  Rodrigo Ponce\\ \noindent 
             \texttt{rponce@utalca.cl}\\
   
			\noindent	Universidad de Talca, Instituto de Matemáticas, Casilla 747, Talca, Chile.\\

			\textit{Key words:} approximate controllability, evolution equation with memory, singular kernel, resolvent family.
			
			Mathematics Subject Classification (2020): 34H05, 45J05, 93B05.

}}}

\begin{document}
	\title[Evolution equation with memory]{Controllability problem of an evolution equation with singular memory \Addresses}
	\author [S. Arora and R. Ponce]{Sumit Arora and Rodrigo Ponce}
    \thanks{This work was supported by Project ANID-Fondecyt 3240359.}
	\maketitle{}

	\begin{abstract} 		
		This work addresses control problems governed by a semilinear evolution equation with singular memory kernel $\kappa(t)=\alpha e^{-\beta t}\frac{t^{\nu-1}}{\Gamma(\nu)}$, where $\alpha>0, \beta\ge 0$, and $0<\nu<1$. We examine the existence of a mild solution and the approximate controllability of both linear and semilinear control systems. To this end, we introduce the concept of a resolvent family associated with the linear evolution equation with memory and develop some of its essential properties. Subsequently, we consider a linear-quadratic regulator problem to determine the optimal control that yields approximate controllability for the linear control system. Furthermore, we derive sufficient conditions for the existence of a mild solution and the approximate controllability of a semilinear system in a super-reflexive Banach space.  Additionally, we present an approximate controllability result within the framework of a general Banach space. Finally, we apply our theoretical findings to investigate the approximate controllability of the heat equation with singular memory.
		
	\end{abstract}

 \section{Introduction}\label{intro}\setcounter{equation}{0}
In this study, we analyzing the existence and approximate controllability of the following evolution equation with singular memory:
\begin{equation}\label{SEq}
	\left\{
	\begin{aligned}
		w'(t)&=\mathrm{A}w(t)+\int_{0}^{t}\kappa(t-s)\mathrm{A}w(s)\mathrm{d}s+\mathrm{B}u(t)+f(t,w(t)), \ t\in J:=(0,T],\\
		w(0)&=\zeta,
	\end{aligned}
	\right.
\end{equation}
where $\mathrm{A}:D(\mathrm{A})\subseteq\mathbb{W}\to\mathbb{W}$ is a closed linear sectorial operator of angle $\theta\in[0,\pi/2)$ with $\overline{D(\mathrm{A})}=\mathbb{W}$ and $\mathbb{W}$ is a Banach space. The linear operator $\mathrm{B}:\mathbb{U}\to\mathbb{W}$ is bounded with $\left\|\mathrm{B}\right\|_{\mathcal{L}(\mathbb{U};\mathbb{W})}=M$ and the control function $u$ belongs to $\mathrm{L}^2(J;\mathbb{U})$, where $\mathbb{U}$ is another Banach space. The function $ f:J\times \mathbb{W} \rightarrow \mathbb{W} $ is defined in the later section. The kernel $\kappa(t)$ is given by 
	$$\kappa(t)=\alpha e^{-\beta t}\frac{t^{\nu-1}}{\Gamma(\nu)}, t>0,$$
	with $\alpha>0, \beta\ge 0$ and $0<\nu<1$.

We notice that this kernel appears, for example, in Maxwell materials in viscoelasticity theory, see for instance \cite[Section 9,
Chapter II]{JP1993} and, in the applications where the operator $\mathrm{A}$ corresponds to a second-order linear and self-adjoint operator defined on a bounded domain $\Omega$ in $\mathbb{R}^N.$

In the past decades, the study of integro-differential systems has received significant interest within the mathematical community. This growing attention is largely due to the ability of these equations to describe various phenomena that differential equations cannot fully capture. A notable example is heat diffusion in materials with memory. The classical heat equation does
not allow a complete understanding of how heat diffuses in memory materials, mainly because in the classical heat equation, it is assumed that changes in the heat source immediately affect the material. One of the first works on the study of these kind of equations dates back to the work of Coleman and Gurtin \cite{BD-ME-1967}, Gurtin and Pipkin \cite{ME-AP-1968}, and Nunziato \cite{JN-1971}, where the authors studied the heat diffusion in materials with fading memory by considering the internal energy and heat flux as functionals of $w$ and $\nabla w,$ ultimately deriving a mathematical model in the form of \eqref{SEq} to describe this phenomenon. 

On the other hand, controllability is a fundamental concept in both engineering and mathematical control theory. It reflects the ability of a solution to a control problem to lead from any given initial state to a desired target state through appropriate controls. The problem of controllability has been extensively studied in various contexts (see for instance, \cite{VB-1993,JMC-2007,EZ-2007} and the references therein). In the context of infinite-dimensional control systems, two fundamental concepts of controllability, say, exact controllability and approximate controllability are distinct and extensively studied. Generally, exact controllability is typically difficult to achieve in infinite-dimensional control systems (cf. \cite{JMC-2007,SG-OY-2013,AH-LP-2012,TR-1980,TR-1977,EZ-2007} etc.). Specifically, the prototype wave equation and transport equation etc., are examples where the exact controllability is attainable.  Thus, it is important to study the problem of approximate controllability, which is a weaker form of controllability . This concept implies that the system can steer into an arbitrary small neighborhood of the final state. In particular, the study of the approximate controllability in infinite-dimensional control systems has received considerable interest due to its large applications (see \cite{JMC-2007,JK-2018,SIAM2003,EZ-2007} for instance). 
Several researchers have obtained significant results on the approximate controllability of various semilinear and nonlinear systems (cf. \cite{SAAK2024,JDE2020,MTM-20} and the references therein).

Therefore, the study of the controllability of integro-differential systems in the form of \eqref{SEq} arises naturally. Controllability problems for parabolic integro-differential equations have been extensively studied in recent years (see for instance \cite{Ba-Ia-00,FWC-XZ-EZ-17,XZ-21} and the references therein). Typical hypotheses in the literature on kernel $\kappa$ to study the controllability of system \eqref{SEq} are to assume that $\kappa$ is a completely monotonic kernel \cite{Ta-Ga-16}, a completely monotonic and infinitely differentiable function \cite{Ba-Ia-00} or a continuously differentiable function \cite{AH-LP-2015,XZ-21}. However, in our case, the relaxation kernel $\kappa$ is not necessarily completely monotonic (see \cite[Section 2.7]{Ar-Ba-Hi-Ne-11}) or a continuously differentiable function, and corresponds to a singular kernel.

 Recently, Pandolfi  \cite{LP-2021} studied the approximate controllability of a general class of control systems with singular memory by applying boundary control. The considered system particularly includes systems with fractional derivatives and integrals as well as the standard heat equation. 

In this work, we consider the distributed control problem involving singular memory and our aim to find conditions on $\kappa$ and $\mathrm{A}$ that guarantee the approximate controllability of the system \eqref{SEq}. To achieve this, we first observe that  if $\nu\geq 1,$ then the equation has a non-singular kernel and the theory of $C_0$-semigroups can be applied (see for instance \cite{KJE2000}), but, our conditions on the parameters $\alpha,\beta$ and $\nu,$ imply that the kernel $\kappa$ does not belong to the space $\mathrm{W}^{1,1}_{\rm loc}(\mathbb{R}_+,\mathbb{C})$ and therefore it is not possible to reduce the system \eqref{SEq} to a first-order equation (see \cite[Chapter VI]{KJE2000}). Thus, it is essential to introduce the theory of resolvent families, which seems to be an appropriate tool to deal the system with singular memory. 

The investigation of approximate controllability in systems with memory using resolvent families is relatively new, and there are only few results concerning the existence and approximate controllability of Volterra-type systems, see for instance \cite{YK-YP-2019,LP-2021,RP2021} etc.

 To study the approximate controllability of \eqref{SEq}, we first examine the well-posedness of the system described in \eqref{NHS}, which is equivalent to the existence of the resolvent family defined below. We also establish key properties of the resolvent family that are essential for subsequent analysis. Further, we consider the control problem and first develop the existence of a mild solution of semilinear equation for a given control $u\in\mathrm{L}^{2}(J;\mathbb{U})$ by assuming the Carath\'eodory condition on the nonlinear term $f(\cdot,\cdot)$ (see Theorems \ref{thm2.9} and \ref{thm2.10}). We then discuss the approximate controllability of the linear problem via the optimal control problem (see Subsection \ref{sublinear}). Next, we establish sufficient conditions for the approximate controllability of the semilinear system \eqref{SEq} in a super-reflexive Banach space (cf. Theorem \ref{thm3.11} and \ref{thm4.4}). We also present the approximate controllability of the semilinear system \eqref{SEq} within the framework of
general Banach spaces by assuming that the corresponding linear system is approximately controllable on $[0,t]$ for all $0<t\le T$ (see Theorem \ref{thm3.15}). Finally, in the applications, we derive an extension of the of the rank condition of the controllability to the system \eqref{SEq} (see Proposition \ref{Prop4.2}).


Throughout this manuscript, the duality pairing between $\mathbb{W}$ and its dual $\mathbb{W^{*}}$ is indicated by $\langle \cdot, \cdot  \rangle $. The notations $\mathcal{L}(\mathbb{U};\mathbb{W})$ and $\mathcal{L}(\mathbb{W})$ denote, respectively, the space of all bounded linear operators from $\mathbb{U}$ to $\mathbb{W}$, equipped with the operator norm $\|\cdot\|_{\mathcal{L}(\mathbb{U};\mathbb{W})}$, and the space of all bounded linear operators on $\mathbb{W}$, endowed with the norm $\|\cdot\|_{\mathcal{L}(\mathbb{W})}$. For a densely defined closed linear operator $\mathrm{A}$ on $\mathbb{W}$, the set $\sigma(\mathrm{A})$ and $\rho(\mathrm{A})$ represent the spectrum and the resolvent set of $\mathrm{A}$, respectively.

\section{Well-posedness and resolvent} \label{Resolvent}\setcounter{equation}{0} 
This section first focus to the well-posedness of the linear non-homogeneous integrodifferential equation of the form:
\begin{equation}\label{NHS}
	\left\{
	\begin{aligned}
		w'(t)&=\mathrm{A}w(t)+\int_{0}^{t}\kappa(t-s)\mathrm{A}w(s)\mathrm{d}s+h(t),\ t\in(0,T],\\
		w(0)&=\zeta,
	\end{aligned}
	\right.
\end{equation}
where the operator $\mathrm{A}$ and the kernel $\kappa(\cdot)$ are the same as in \eqref{SEq}. $h:J\to\mathbb{W}$ is an appropriate function.

To complete this, we introduce the concept of a \emph{resolvent family} corresponding to the homogeneous linear evolution equation:
\begin{equation}\label{LEq}
	\left\{
	\begin{aligned}
		w'(t)&=\mathrm{A}w(t)+\int_{0}^{t}\kappa(t-s)\mathrm{A}w(s)\mathrm{d}s, \ t\in(0,T],\\
		w(0)&=\zeta.
	\end{aligned}
	\right.
\end{equation}
 It is straightforward to identify that the system \eqref{LEq} is equivalent to the integral equation: 
\begin{align}\label{LIEq}
    w(t)=\zeta+\int^t_0(1+1*\kappa)(t-s)\mathrm{A}w(s)\mathrm{ds},\ \ t\in[0, T].
\end{align}
The well-posedness of the system \eqref{LEq} (or corresponding integral equation \eqref{LIEq}) is equivalent to the existence of a resolvent family, that is, a family of bounded linear operators $(\mathscr{G}_{\alpha,\beta}^\nu(t))_{t\geq 0}$ on $\mathbb{W}$, which satisfy the following properties (see, \cite{JP1993}, Chapter 1):
\begin{itemize}
    \item [\textit{(a)}] $\mathscr{G}_{\alpha,\beta}^\nu(t)$ is strongly continuous on $\mathbb{R}_+$ and $\mathscr{G}_{\alpha,\beta}^\nu(0)=\mathrm{I}$;
    \item [\textit{(b)}] $\mathscr{G}_{\alpha,\beta}^\nu(t)z\in D(\mathrm{A})$ and $\mathscr{G}_{\alpha,\beta}^\nu(t)\mathrm{A}z=\mathrm{A}\mathscr{G}_{\alpha,\beta}^\nu(t)z$ for all $z\in D(\mathrm{A})$ and $t\geq 0$;
    \item [\textit{(c)}] For all $z\in D(\mathrm{A})$ and $t\geq 0$, the resolvent equation satisfy:
   \begin{align}\label{REq}
    \mathscr{G}_{\alpha,\beta}^\nu(t)z=z+\int^t_0(1+1*\kappa)(t-s)\mathrm{A}\mathscr{G}_{\alpha,\beta}^\nu(s)z\mathrm{ds}.
\end{align}
\end{itemize}
The above family is called an $(\alpha,\beta,\nu)$-\emph{resolvent family} generated by $\mathrm{A}.$ If the family exists, then for given $\zeta\in\mathbb{W}$, the unique mild solution of the system \eqref{LIEq} is given by (cf. \cite{RP2021}):
$$w(t)=\mathscr{G}_{\alpha,\beta}^\nu(t)\zeta, \ \ t\in[0, T].$$

If the Laplace transformation of $\mathscr{G}_{\alpha,\beta}^\nu(\cdot)$ exists, then it will be given by
$$\widehat{\mathscr{G}}_{\alpha,\beta}^\nu(\lambda)=(\lambda I-\frac{\alpha}{(\lambda+\beta)^{\nu}}\mathrm{A}-\mathrm{A})^{-1}\in\mathcal{L}(\mathbb{W}).$$

In order to prove the existence of the family $\mathscr{G}_{\alpha,\beta}^\nu(\cdot)$, we first introduce the definition of a sectorial operator. An operator $\mathrm{A}$ is said to be sectorial (of angle $\theta$), if there exists $0<\theta\le \frac{\pi}{2}$ such that the sector
\begin{align*}
    \Sigma_\theta:=\{ \lambda\in\mathbb{C}\setminus\{0\}: |\arg(\lambda)|< \theta+\frac{\pi}{2}\}\subset \rho(\mathrm{A}),
\end{align*}
and if for each $\omega\in(0,\theta)$, there exists $K_{\omega}\ge 1$ such that
\begin{align*}
    ||(\lambda\mathrm{I}-\mathrm{A})^{-1}||\le\frac{K_{\omega}}{|\lambda|}, \ \mbox{for all}\ 0\neq \lambda\in\overline{\Sigma}_{\theta-\omega},
\end{align*}
holds. For more details, the interested readers are recommended to see \cite{KJE2000,MH2006} and the references therein.

 In the sequel, we define a set $\Sigma_{r,\theta}=\{\lambda\in\mathrm{\textbf{C}}\setminus\{0\}:|\lambda|>r, |\arg(\lambda)|<\frac{\pi}{2}+\theta\}$ for $r>0$ and $\theta\in(0,\pi/2)$. Let $\Gamma_{r,\theta}=\cup_{k=1}^{3} \Gamma^{k}_{r,\theta}$, where $\Gamma^{1}_{r,\theta}=\{se^{i(\frac{\pi}{2}+\theta)}:s\ge r\},\ \Gamma^{2}_{r,\theta}=\{re^{i\phi}: -(\frac{\pi}{2}+\theta)\le\phi\le(\frac{\pi}{2}+\theta)\},\ \Gamma^{3}_{r,\theta}=\{se^{-i(\frac{\pi}{2}+\theta)}:s\ge r\}$ and the orientation is taken counterclockwise. In addition, let $$\Omega(\widehat{\mathscr{G}}_{\alpha,\beta}^\nu)=\left\{\lambda\in\mathrm{\textbf{C}}:\widehat{\mathscr{G}}_{\alpha,\beta}^\nu(\lambda):=(\lambda\mathrm{I}-\frac{\alpha}{(\lambda+\beta)^{\nu}}\mathrm{A}-\mathrm{A})^{-1}\in\mathcal{L}(\mathbb{W})\right\}.$$ 
\begin{lem}\label{lem2.1}
    Let $\mathrm{A}$ be a sectorial operator with angle $\frac{\nu\pi}{2}$. Then there exists an $\tilde{r}>0$ such that $\Sigma_{\tilde{r},\frac{\nu\pi}{2}}\in\Omega(\widehat{\mathscr{G}}_{\alpha,\beta}^\nu)$ and the mapping $\widehat{\mathscr{G}}_{\alpha,\beta}^\nu:\Sigma_{\tilde{r},\frac{\nu\pi}{2}}\to\mathcal{L}(\mathbb{W})$ is analytic. Moreover, 
     \begin{align}\label{r1}
     \widehat{\mathscr{G}}_{\alpha,\beta}^\nu(\lambda)=(\lambda\mathrm{I}-\mathrm{A})^{-1}(\mathrm{I}-\frac{\alpha}{(\lambda+\beta)^{\nu}}\mathrm{A}(\lambda\mathrm{I}-\mathrm{A})^{-1})^{-1},\end{align}
     and there exists a constant $\tilde{M}>0$ such that \begin{align}\label{re}||\lambda\widehat{\mathscr{G}}_{\alpha,\beta}^\nu(\lambda)||_{\mathcal{L}(\mathbb{W})}\le \tilde{M}.\end{align}
\end{lem}
\begin{proof}
    Since $$\left\|\frac{\alpha}{(\lambda+\beta)^{\nu}}\mathrm{A}(\lambda\mathrm{I}-\mathrm{A})^{-1}\right\|_{\mathcal{L}(\mathbb{W})}=\left\|\frac{\alpha}{(\lambda+\beta)^{\nu}}(\mathrm{I}-\lambda(\lambda\mathrm{I}-\mathrm{A})^{-1})\right\|_{\mathcal{L}(\mathbb{W})}\le \frac{|\alpha|C}{|(\lambda+\beta)^{\nu}|}.$$
     This fact implies that there exists a $\tilde{r}>0$ such that $\left\|\frac{\alpha}{(\lambda+\beta)^{\nu}}\mathrm{A}(\lambda\mathrm{I}-\mathrm{A})^{-1}\right\|_{\mathcal{L}(\mathbb{W})}<1$ for $\lambda\in\Sigma_{\tilde{r},\frac{\nu\pi}{2}}.$
    Subsequently, the operator $(\mathrm{I}-\frac{\alpha}{(\lambda+\beta)^{\nu}}\mathrm{A}(\lambda\mathrm{I}-\mathrm{A})^{-1})^{-1}$ is invertible and has a continuous inverse. Moreover, for $x\in\mathbb{W}$, we have 
\begin{align*}
 &(\lambda\mathrm{I}-\frac{\alpha}{(\lambda+\beta)^{\nu}}\mathrm{A}-\mathrm{A})(\lambda\mathrm{I}-\mathrm{A})^{-1}(\mathrm{I}-\frac{\alpha}{(\lambda+\beta)^{\nu}}\mathrm{A}(\lambda\mathrm{I}-\mathrm{A})^{-1})^{-1}x\\&=(\lambda\mathrm{I}-\frac{\alpha}{(\lambda+\beta)^{\nu}}\mathrm{A}-\mathrm{A})(\lambda\mathrm{I}-\frac{\alpha}{(\lambda+\beta)^{\nu}}\mathrm{A}-\mathrm{A})^{-1}x=x,
 \end{align*}
 and for $x\in D(\mathrm{A})$, we verify
 \begin{align*}
     &(\lambda\mathrm{I}-\mathrm{A})^{-1}(\mathrm{I}-\frac{\alpha}{(\lambda+\beta)^{\nu}}\mathrm{A}(\lambda\mathrm{I}-\mathrm{A})^{-1})^{-1}(\lambda\mathrm{I}-\frac{\alpha}{(\lambda+\beta)^{\nu}}\mathrm{A}-\mathrm{A})x\\&=(\lambda\mathrm{I}-\frac{\alpha}{(\lambda+\beta)^{\nu}}\mathrm{A}-\mathrm{A})^{-1}(\lambda\mathrm{I}-\frac{\alpha}{(\lambda+\beta)^{\nu}}\mathrm{A}-\mathrm{A})x\\&=x.
 \end{align*}
    Hence,
    $$\widehat{\mathscr{G}}_{\alpha,\beta}^\nu(\lambda)=(\lambda\mathrm{I}-\frac{\alpha}{(\lambda+\beta)^{\nu}}\mathrm{A}-\mathrm{A})^{-1}=(\lambda\mathrm{I}-\mathrm{A})^{-1}(\mathrm{I}-\frac{\alpha}{(\lambda+\beta)^{\nu}}\mathrm{A}(\lambda\mathrm{I}-\mathrm{A})^{-1})^{-1}.$$ 
 As $||(\mathrm{I}-\frac{\alpha}{(\lambda+\beta)^{\nu}}\mathrm{A}(\lambda\mathrm{I}-\mathrm{A})^{-1})^{-1}||_{\mathcal{L}(\mathbb{W})}<1$, from the above expression, we deduce that
 \begin{align*}
    \widehat{\mathscr{G}}_{\alpha,\beta}^\nu(\lambda)=(\lambda \mathrm{I}-\mathrm{A})^{-1}\sum_{n=0}^{\infty}\left(\frac{\alpha}{(\lambda+\beta)^{\nu}}\mathrm{A}(\lambda\mathrm{I}-\mathrm{A})^{-1}\right)^{n}\ \mbox{for}\ \lambda\in\Sigma_{\tilde{r},\frac{\nu\pi}{2}}.
 \end{align*}
 Thus, the function $\widehat{\mathscr{G}}_{\alpha,\beta}^\nu(\lambda)$ is analytic for all $\lambda\in\Sigma_{\tilde{r},\frac{\nu\pi}{2}}$. 


Now since $\mathrm{A}$ is a sectorial operator of angle $\frac{\nu\pi}{2}$, then we compute
 \begin{align*}
 ||\lambda\widehat{\mathscr{G}}_{\alpha,\beta}^\nu(\lambda)||_{\mathcal{L}(\mathbb{W})}&=||(\lambda\mathrm{I}-\mathrm{A})^{-1}(\mathrm{I}-\frac{\alpha}{(\lambda+\beta)^{\nu}}\mathrm{A}(\lambda\mathrm{I}-\mathrm{A})^{-1})^{-1}||_{\mathcal{L}(\mathbb{W})}\\&\le K||(\mathrm{I}-\frac{\alpha}{(\lambda+\beta)^{\nu}}\mathrm{A}(\lambda\mathrm{I}-\mathrm{A})^{-1})^{-1}||_{\mathcal{L}(\mathbb{W})}\\&\le \tilde{M}\ \mbox{for all}\  \lambda\in\Sigma_{\tilde{r},\frac{\nu\pi}{2}},
 \end{align*}
 the proof is complete.
\end{proof}
Using the inversion formula of the Laplace transformation, we define an operator family $(\mathscr{G}_{\alpha,\beta}^\nu(t))_{t\geq 0}$ by
\begin{equation}\label{resolvent}
		\mathscr{G}_{\alpha,\beta}^\nu(t)=
		\begin{dcases}
			\frac{1}{2\pi i}\int_{\Gamma_{\tilde{r},\frac{\nu\pi}{2}}}e^{\lambda t}\widehat{\mathscr{G}}_{\alpha,\beta}^\nu(\lambda)\mathrm{d}\lambda, \ t>0,\\
			\mathrm{I}, \ \qquad\qquad\qquad\qquad\quad t=0.
		\end{dcases}
	\end{equation}
    We will now verify that  $(\mathscr{G}_{\alpha,\beta}^\nu(t))_{t\geq 0}$  is a \emph{resolvent family}  for \eqref{LIEq}.
    \begin{lem}
        The mapping $\mathscr{G}_{\alpha,\beta}^\nu(\cdot)$ is exponentially bounded in $\mathcal{L}(\mathbb{W})$.
    \end{lem}
    \begin{proof}
The proof of the lemma follows by the direct calculation. For $t>0$, let us evaluate
 \begin{align*}
     ||\mathscr{G}_{\alpha,\beta}^\nu(t)||_{\mathcal{L}(\mathbb{W})}&\le \frac{\tilde{M}}{2\pi}\int_{\Gamma_{\tilde{r},\frac{\nu\pi}{2}}}\frac{|e^{\lambda t}|}{|\lambda|}|\mathrm{d}\lambda|\nonumber\\&=\sum_{i=1}^{3}\frac{\tilde{M}}{2\pi}\int_{\Gamma^{i}_{\tilde{r},\frac{\nu\pi}{2}}}\frac{e^{\mathrm{Re}(\lambda)t}}{|\lambda|}|\mathrm{d}\lambda|.
 \end{align*}
Let us first compute
\begin{align*}
    \frac{\tilde{M}}{2\pi}\int_{\Gamma^{1}_{\tilde{r},\frac{\nu\pi}{2}}}\frac{e^{\mathrm{Re}(\lambda)t}}{|\lambda|}|\mathrm{d}\lambda|=\frac{\tilde{M}}{2\pi}\int_{r}^{\infty}\frac{e^{-st\sin(\frac{\nu\pi}{2})}}{s}\mathrm{d}s=\frac{\tilde{M}}{2\pi}\int_{rt\sin(\frac{\nu\pi}{2})}^{\infty}\frac{e^{-l}}{l}\mathrm{d}l\le C_{1}.
\end{align*}
The estimate over $\Gamma^{3}_{\tilde{r},\frac{\nu\pi}{2}}$ is evaluated in a similar manner. We now calculate
\begin{align*}
    \frac{\tilde{M}}{2\pi}\int_{\Gamma^{2}_{\tilde{r},\frac{\nu\pi}{2}}}\frac{e^{\mathrm{Re}(\lambda)t}}{|\lambda|}|\mathrm{d}\lambda|&=\frac{\tilde{M}}{2\pi}\int_{-(1+\nu)\frac{\pi}{2}}^{(1+\nu)\frac{\pi}{2}}\frac{e^{rt\cos(\phi)}}{r}\mathrm{d}\phi\nonumber\\&\le \frac{\tilde{M}(1+\nu)}{2 r} e^{rt}.
\end{align*}
Therefore, the mapping $\mathscr{G}_{\alpha,\beta}^\nu(\cdot)$ is exponentially bounded in $\mathcal{L}(\mathbb{W})$. 
     \end{proof}
     \begin{lem}
        The mapping $\mathscr{G}_{\alpha,\beta}^\nu(\cdot):[0,\infty)\to\mathcal{\mathbb{W}}$ is strongly continuous.  
    \end{lem}
    \begin{proof}
        It is clear from \eqref{resolvent} and the applications of the Lebesgue dominated convergence theorem that the mapping $\mathscr{G}_{\alpha,\beta}^\nu(\cdot)w$ is continuous for each $t>0$ and every $w\in\mathbb{W}$. Next, we verify the continuity argument at $t=0$. In fact, using the fact $$\frac{1}{2\pi i}\int_{\Gamma_{\tilde{r},\frac{\nu\pi}{2}}}\frac{e^{\lambda t}}{\lambda}\mathrm{d}\lambda=1,$$ for any $0<t\le 1$ and $w\in D(\mathrm{A})$, we can write
        \begin{align*}
           \mathscr{G}_{\alpha,\beta}^\nu(t)w-w&=\frac{1}{2\pi i}\int_{\Gamma_{\tilde{r},\frac{\nu\pi}{2}}}\left(e^{\lambda t}\widehat{\mathscr{G}}_{\alpha,\beta}^\nu(\lambda)w-\frac{e^{\lambda t}}{\lambda}w\right)\mathrm{d}\lambda\nonumber\\&=\frac{1}{2\pi i}\int_{\Gamma_{\tilde{r},\frac{\nu\pi}{2}}}e^{\lambda t}\lambda^{-1}\widehat{\mathscr{G}}_{\alpha,\beta}^\nu(\lambda)\left(\frac{\alpha}{(\lambda+\beta)^{\nu}}\mathrm{A}+\mathrm{A}\right)w\mathrm{d}\lambda.
        \end{align*}
        Using the bound in \eqref{re}, we obtain
        \begin{align*}
            \left\|e^{\lambda t}\lambda^{-1}\widehat{\mathscr{G}}_{\alpha,\beta}^\nu(\lambda)\left(\frac{\alpha}{(\lambda+\beta)^{\nu}}\mathrm{A}+\mathrm{A}\right)w\right\|_{\mathbb{W}}\le e^{r}C\left(\frac{|\alpha|}{|\lambda+\beta|^\nu|\lambda|^{2}}+\frac{1}{|\lambda|^2}\right)||w||_{1},
        \end{align*}
        for $\lambda\in\Gamma_{\tilde{r},\frac{\nu\pi}{2}}$, where $||\cdot||_{1}$ denotes the graph norm. The above fact and the Lebesgue dominated convergence theorem imply that
        \begin{align*}
            \lim_{t\to 0^+}\left(\mathscr{G}_{\alpha,\beta}^\nu(t)w-w\right)=\frac{1}{2\pi i}\int_{\Gamma_{\tilde{r},\frac{\nu\pi}{2}}}\lambda^{-1}\widehat{\mathscr{G}}_{\alpha,\beta}^\nu(\lambda)\left(\frac{\alpha}{(\lambda+\beta)^{\nu}}\mathrm{A}+\mathrm{A}\right)w\mathrm{d}\lambda.
        \end{align*}
        Now let $\Gamma_{L,\frac{\nu\pi}{2}}$ be the curve $Le^{i\xi}$ for $-\frac{(1+\nu)\pi}{2}\le\xi\le\frac{(1+\nu)\pi}{2}$. Then we obtain
        \begin{align}\label{e2.7}
       \left\|\int_{\Gamma_{L,\frac{\nu\pi}{2}}}\lambda^{-1}\widehat{\mathscr{G}}_{\alpha,\beta}^\nu(\lambda)\left(\frac{\alpha}{(\lambda+\beta)^{\nu}}\mathrm{A}+\mathrm{A}\right)w\mathrm{d}\lambda\right\|_{\mathbb{W}}\le C\frac{\nu\pi}{2}\left(\frac{|\alpha|}{(L-\beta)^\nu L} +\frac{1}{L}\right)|w||_{1}.
        \end{align}
        Also, it follows from the Cauchy's theorem 
        \begin{align*}
            &\frac{1}{2\pi i}\int_{\Gamma_{\tilde{r},\frac{\nu\pi}{2}}}\lambda^{-1}\widehat{\mathscr{G}}_{\alpha,\beta}^\nu(\lambda)\left(\frac{\alpha}{(\lambda+\beta)^{\nu}}\mathrm{A}+\mathrm{A}\right)w\mathrm{d}\lambda\\&=\lim_{L\to\infty}\frac{1}{2\pi i}\int_{\Gamma_{L,\frac{\nu\pi}{2}}}\lambda^{-1}\widehat{\mathscr{G}}_{\alpha,\beta}^\nu(\lambda)\left(\frac{\alpha}{(\lambda+\beta)^{\nu}}\mathrm{A}+\mathrm{A}\right)w\mathrm{d}\lambda.
        \end{align*}
        Combining this equality with the inequality \eqref{e2.7}, we arrive $$\lim_{t\to 0^+}\left(\mathscr{G}_{\alpha,\beta}^\nu(t)w-w\right)=0 \ \mbox{for all}\ x\in D(\mathrm{A}).$$
Therefore, the mapping $\mathscr{G}_{\alpha,\beta}^\nu(\cdot)w$ is continuous on $\mathcal{L}(\mathbb{W})$ as $D(\mathrm{A})$ is dense and $\mathscr{G}_{\alpha,\beta}^\nu(\cdot)$ bounded on $[0,1]$.
\end{proof}
Using Lemma \ref{lem2.1} and the properties of the Laplace transformation we deduce that $\mathscr{G}_{\alpha,\beta}^\nu(t)w$ satisfies the resolvent equation \eqref{REq} for all $w\in D(\mathrm{A})$ and $t\geq 0.$ Thus, $(\mathscr{G}_{\alpha,\beta}^\nu(t))_{t\ge 0}$ is the $(\alpha,\beta,\nu)$-\emph{resolvent family} generated by $\mathrm{A}.$ Let us suppose $\sup_{t\in[0,T]}\|\mathscr{G}_{\alpha,\beta}^\nu(t)\|_{\mathcal{L}(\mathbb{W})}\le N$ throughout the paper.

We now examine some important lemmas related to the resolvent family $(\mathscr{G}_{\alpha,\beta}^\nu(t))_{t\ge 0}$, which are essential to develop our results. The proofs of these lemmas closely follow the techniques outlined in \cite{SAAK2024,JIEA2011}. 
\begin{lem}\label{lem2.7}
	The operator $\mathscr{G}_{\alpha,\beta}^\nu(t)$ is continuous in $\mathcal{L}(\mathbb{W})$ for $t>0$.
\end{lem}
\begin{lem}\label{lem2.6}
	If $\mathrm{R}(\lambda_0,\mathrm{A})$ is compact for some $\lambda_0\in\rho(\mathrm{A})$, then the operator $\mathscr{G}_{\alpha,\beta}^\nu(t)$ is compact for all $t>0$.
\end{lem}
\begin{lem}\label{lem2.6}
	Let $\mathscr{G}_{\alpha,\beta}^\nu(t)$ be compact for $t>0$. Then then the following conditions hold:
	\begin{itemize}
		\item [(a)]$\lim\limits_{h\to 0^+}\|\mathscr{G}_{\alpha,\beta}^\nu(t+h)-\mathscr{G}_{\alpha,\beta}^\nu(h)\mathscr{G}_{\alpha,\beta}^\nu(t)\|_{\mathcal{L}(\mathbb{W})}=0$ for all $t>0$.
		\item [(b)]$\lim\limits_{h\to 0^+}\|\mathscr{G}_{\alpha,\beta}^\nu(t)-\mathscr{G}_{\alpha,\beta}^\nu(h)\mathscr{G}_{\alpha,\beta}^\nu(t-h)\|_{\mathcal{L}(\mathbb{W})}=0$ for all $t>0$.
	\end{itemize} 
\end{lem}
It is immediate to present the following existence result of the non-homogeneous equation.
\begin{pro}\label{pre2.7}
    Suppose the operator $\mathrm{A}$ is sectorial with angle $\frac{\nu\pi}{2}$ and the function $h:J\to\mathbb{W}$ is continuously differentiable. Then for each $\zeta\in D(\mathrm{A})$, the equation:
\begin{equation}\label{NLEq}
	\left\{
	\begin{aligned}
		w'(t)&=\mathrm{A}w(t)+\int_{0}^{t}\kappa(t-s)\mathrm{A}w(s)\mathrm{d}s+h(t),\ t\in(0,T],\\
		w(0)&=\zeta,
	\end{aligned}
	\right.
\end{equation}
admits a unique \emph{classical solution} which is give by
\begin{align*}
	w(t)&=\mathscr{G}_{\alpha,\beta}^\nu(t)\zeta+\int_{0}^{t}\mathscr{G}_{\alpha,\beta}^\nu(t-s)h(s), t\in J.
\end{align*}
On the other hand, if $\zeta\in\mathbb{W}$ and the function $h\in\mathrm{L}^1(J;\mathbb{W})$, then the above function $w\in C(J;\mathbb{W})$ is called the \emph{mild solution} of the system \eqref{NLEq}.
\end{pro}
Let us now introduce and investigate the existence of a mild solution for the semilinear control problem given in \eqref{SEq}. The following restrictions are imposed on the nonlinear function $f(\cdot,\cdot)$:
\begin{itemize}
    \item [\textbf{\textit{(F1)}}]  The function $f(t,\cdot): \mathbb{W}\rightarrow\mathbb{W}$ is continuous for a.e. $t\in J$. The map $t\mapsto f(t,w) $ is strongly measurable on $J$ for each $ w\in \mathbb{W}$. 
    \item [\textbf{\textit{(F2)}}] There exists a function $\gamma\in\mathrm{L}^1(J;\mathbb{R}^+)$ such that 
			$$\|f(t,w)\|_{\mathbb{W}}\le\gamma(t), \ \mbox{for a.e.}\ t\in J \ \mbox{and for all} \ w\in\mathbb{W}.$$
\end{itemize}
\begin{Def}
    For a given $u\in\mathrm{L}^{2}(J;\mathbb{U})$, a function $w\in C(J;\mathbb{W})$ is called a mild solution of the problem \eqref{SEq}, if it is satisfy the following  integral equation:
    \begin{align}\label{Seq1}
	w(t)&=\mathscr{G}_{\alpha,\beta}^\nu(t)\zeta+\int_{0}^{t}\mathscr{G}_{\alpha,\beta}^\nu(t-s)\left[\mathrm{B}u(s)+f(s,w(s)\right]\mathrm{d}s, t\in J.
\end{align}
\end{Def}
Throughout this text, we assume that the operator $\mathrm{A}$ is a sectorial operator with angle $\frac{\nu\pi}{2}$ and the operator $\mathrm{R}(\lambda_0,\mathrm{A})$ is compact for some $\lambda_0\in\rho(\mathrm{A})$.

Existence of a mild solution for the system \eqref{SEq} is presented below by using the \emph{Schauder fixed point theorem}. 
\begin{theorem}\label{thm2.9}
    Let Assumptions \textbf{\textit{(F1)}}-\textbf{\textit{(F2)}} hold true. Then for each $u\in\mathrm{L}^{2}(J;\mathbb{U})$, the system \eqref{SEq} has at least one mild solution on $J$.
\end{theorem}
\begin{proof}
    For a given $u\in\mathrm{L}^{2}(J;\mathbb{U})$, we define an operator $\mathcal{P}:C(J;\mathbb{W})\to C(J;\mathbb{W})$ such that
		\begin{align*}
			(\mathcal{P}w)(t)=\mathscr{G}_{\alpha,\beta}^\nu(t)\zeta+\int_{0}^{t}\mathscr{G}_{\alpha,\beta}^\nu(t-s)\left[\mathrm{B}u(s)+f(s,w(s))\right]\mathrm{d}s, t\in J.
		\end{align*}
		It is clear that the system \eqref{SEq} possesses a mild solution, if the operator $\mathcal{P}$ admits a fixed point. The proof of the operator $\mathcal{P}$ has a fixed point is divided into several steps.
			\vskip 0.1in 
		\noindent\textbf{Step (1): } In first step, we will show that the map $\mathcal{P}$ maps to the set $Q_r=\{w\in C(J;\mathbb{W}): \left\|x\right\|_{C(J;\mathbb{W})}\le r\}$ into itself for some $r>0$. For this, let us estimate
\begin{align*}
    \|(\mathcal{P}w)(t)\|_{\mathbb{W}}&\le\|\mathscr{G}_{\alpha,\beta}^\nu(t)\zeta\|_{\mathbb{W}}+\int_{0}^{t}\|\mathscr{G}_{\alpha,\beta}^\nu(t-s)\left[\mathrm{B}u(s)+f(s,w(s))\right]\|_{\mathbb{W}}\mathrm{d}s\nonumber\\&\le N\|\zeta\|_{\mathbb{W}}+NM\int_0^t\|u(s)\|_{\mathbb{U}}\mathrm{d}s+N\int_0^t \gamma(s)\mathrm{d}s\nonumber\\&\le N\|\zeta\|_{\mathbb{W}}+NMT^{1/2}\|u\|_{\mathrm{L}^2(J;\mathbb{U})}+N\|\gamma\|_{\mathrm{L}^1(J;\mathbb{R}^+)}=r.
\end{align*}
From the above inequality, we infer that there exists an $r>0$ such that $\mathcal{P}(Q_r)\subseteq Q_r$. 
\vskip 0.1in 
		\noindent\textbf{Step (2): } In this step, we claim that the operator $\mathcal{P}$ is continuous. To accomplish the claim, we consider a sequence $\{{w}^n\}^\infty_{n=1}\subseteq Q_r$ such that ${w}^n\rightarrow {w}\ \mbox{in}\ Q_r,$ that is,
		$$\lim\limits_{n\rightarrow \infty}\left\|w^n-w\right\|_{C(J;\mathbb{W})}=0.$$
Using Assumption \textbf{\textit{(F1)}}-\textbf{\textit{(F2)}} along with Lebesgue's dominate convergence theorem, we obtain 
        \begin{align*}
	\left\|(\mathcal{P}w^{n})(t)-(\mathcal{P}w)(t)\right\|_{\mathbb{W}}&\leq\left\|\int_{0}^{t}\mathscr{G}_{\alpha,\beta}^\nu(t-s)\left[f(s, w^n(s))-f(s,w(s))\right]\mathrm{d}s\right\|_{\mathbb{W}}\nonumber\\&\le N\int_{0}^{t}\left\|f(s,w^{n}(s))-f(s,w(s))\right\|_{\mathbb{W}}\mathrm{d}s\nonumber\\&\to 0 \ \text{ as }\ n\to\infty,
\end{align*}
for each $t\in J$. Hence, the map $\mathcal{P}$ is continuous.
\vskip 0.1in 
		\noindent\textbf{Step (3): } Finally, we show that the operator $\mathcal{P}$ is compact. To complete this, we first prove that $\mathcal{P}(Q_r)$ is equicontinuous. To proceed this, let $0\le t_1\le t_2\le T$ and any $w\in Q_r$, we calculate
        \begin{align*}
	&\left\|(\mathcal{P}w)(t_2)-(\mathcal{P}w)(t_1)\right\|_{\mathbb{W}}\nonumber\\&\leq\left\|\left[\mathscr{G}_{\alpha,\beta}^\nu(t_2)-\mathscr{G}_{\alpha,\beta}^\nu(t_1)\right]\zeta\right\|_{\mathbb{W}}+\left\|\int_{t_{1}}^{t_{2}}\mathscr{G}_{\alpha,\beta}^\nu(t_2-s)\left[\mathrm{B}u(s)+f(s,w(s))\right]\mathrm{d}s\right\|_{\mathbb{W}}\nonumber\\&\quad+\left\|\int_{0}^{t_{1}}\left[\mathscr{G}_{\alpha,\beta}^\nu(t_2-s)-\mathscr{G}_{\alpha,\beta}^\nu(t_1-s)\right]\left[\mathrm{B}u(s)+f(s,w(s))\right]\mathrm{d}s\right\|_{\mathbb{W}}\nonumber\\&\leq\left\|\left[\mathscr{G}_{\alpha,\beta}^\nu(t_2)-\mathscr{G}_{\alpha,\beta}^\nu(t_1)\right]\zeta\right\|_{\mathbb{W}}+NM(t_2-t_1)^{1/2}\|u\|_{\mathrm{L}^2([t_1,t_2];\mathbb{U})}+N\int_{t_1}^{t_2}\gamma(s)\mathrm{d}s\nonumber\\&\quad+\int_{0}^{t_{1}}\left\|\left[\mathscr{G}_{\alpha,\beta}^\nu(t_2-s)-\mathscr{G}_{\alpha,\beta}^\nu(t_1-s)\right]\left[\mathrm{B}u(s)+f(s,w(s))\right]\right\|_{\mathbb{W}}\mathrm{d}s.
\end{align*}
	If $t_1=0$, then we get
\begin{align*}
	\lim_{t_2\to 0^+}\left\|(\mathcal{P}w)(t_2)-(\mathcal{P}w)(t_1)\right\|_{\mathbb{W}}=0,\ \mbox{ uniformly for}\  w\in Q_r.
\end{align*}
For given $\epsilon>0$, let us take  $\epsilon<t_1<T$, we have
\begin{align}\label{e4.8}
	&\left\|(\mathcal{P}w)(t_2)-(\mathcal{P}w (t_1)\right\|_{\mathbb{W}} \nonumber\\&\leq\left\|\left[\mathscr{G}_{\alpha,\beta}^\nu(t_2)-\mathscr{G}_{\alpha,\beta}^\nu(t_1)\right]\zeta\right\|_{\mathbb{W}}+NM(t_2-t_1)^{1/2}\|u\|_{\mathrm{L}^2([t_1,t_2];\mathbb{U})}+N\int_{t_1}^{t_2}\gamma(s)\mathrm{d}s\nonumber\\&\quad+\int_{0}^{t_{1-\epsilon}}\left\|\left[\mathscr{G}_{\alpha,\beta}^\nu(t_2-s)-\mathscr{G}_{\alpha,\beta}^\nu(t_1-s)\right]\left[\mathrm{B}u(s)+f(s,w(s))\right]\right\|_{\mathbb{W}}\mathrm{d}s\nonumber\\&\quad+\int_{t_1-\epsilon}^{t_{1}}\left\|\left[\mathscr{G}_{\alpha,\beta}^\nu(t_2-s)-\mathscr{G}_{\alpha,\beta}^\nu(t_1-s)\right]\left[\mathrm{B}u(s)+f(s,w(s))\right]\right\|_{\mathbb{W}}\mathrm{d}s\nonumber\\&\leq\left\|\left[\mathscr{G}_{\alpha,\beta}^\nu(t_2)-\mathscr{G}_{\alpha,\beta}^\nu(t_1)\right]\zeta\right\|_{\mathbb{W}}+NM(t_2-t_1)^{1/2}\|u\|_{\mathrm{L}^2([t_1,t_2];\mathbb{U})}+N\int_{t_1}^{t_2}\gamma(s)\mathrm{d}s\nonumber\\&\quad+\sup_{s\in[0,t_1-\epsilon]}\left\|\mathscr{G}_{\alpha,\beta}^\nu(t_2-s)-\mathscr{G}_{\alpha,\beta}^\nu(t_1-s)\right\|_{\mathcal{L}(\mathbb{W})}\int_{0}^{t_{1}-\epsilon}\left\|\left[\mathrm{B}u(s)+f(s,w(s))\right]\right\|_{\mathbb{W}}\mathrm{d}s\nonumber\\&\quad+2NM\epsilon^{1/2}\|u\|_{\mathrm{L}^2([t_1-\epsilon, t_1];\mathbb{U})}+2N\int_{t_1-\epsilon}^{t_{1}}\gamma(s)\mathrm{d}s.
\end{align} 
Using the continuity of $\mathscr{G}_{\alpha,\beta}^\nu(t)$ for $t>0$ under the uniform operator topology and  the fact that $\epsilon$ can be chosen arbitrarily small, the right hand side of the expression \eqref{e4.8} converges to zero as $|t_2-t_1| \rightarrow 0$. Consequently, $\mathcal{P}(Q_r)$ is equicontinuous.

Furthermore, we assert that for each $t\in J$, the set $\mathfrak{A}(t)=\{(\mathcal{P}w)(t):w\in Q_r\}$ is relatively compact. When $t=0$, it is straightforward to verify the set  $\mathfrak{A}(t)$  is relatively compact in $Q_r$. Now for a fixed  $ 0<t\leq T$ and let $\eta$ be given with $ 0<\eta<t$. We define
\begin{align*}
	(\mathcal{P}^{\eta}w)(t)=&\mathscr{G}_{\alpha,\beta}^\nu(\eta)\Big[\mathscr{G}_{\alpha,\beta}^\nu(t-\eta)\zeta+\int_{0}^{t-\eta}\mathscr{G}_{\alpha,\beta}^\nu(t-s-\eta)\left[\mathrm{B}u(s)+f(s,w(s))\right]\mathrm{d}s\Big].
\end{align*}
	Thus, the set  $\mathfrak{A}_{\eta}(t)=\{(\mathcal{P}^{\eta}w)(t):w\in Q_r\}$ is relatively compact in $\mathbb{W}$, which follows from the compactness of the operator $\mathscr{G}_{\alpha,\beta}^\nu(\eta)$ (since $\mathrm{R}(\lambda_0,\mathrm{A})$ is compact for some $\lambda_0\in\rho(\mathrm{A})$). As a result, there exist finitely many elements $ w_{i}$'s, for $i=1,\dots, n $ in $ \mathbb{W} $ such that 
		\begin{align*}
	\mathfrak{A}_{\eta}(t) \subset \bigcup_{i=1}^{n}\mathcal{S}_{w_i}(\varepsilon/2),
	\end{align*}
	for some $\varepsilon>0$. Using Lemma \ref{lem2.6}, we can choose small $\eta>0$ such that 
	\begin{align*}
		&\left\|(\mathcal{P}w)(t)-(\mathcal{P}^{\eta}w)(t)\right\|_{\mathbb{W}}\nonumber\\&\le\left\|\left[\mathscr{G}_{\alpha,\beta}^\nu(t)-\mathscr{G}_{\alpha,\beta}^\nu(\eta)\mathscr{G}_{\alpha,\beta}^\nu(t-\eta)\right]\zeta\right\|_{\mathbb{W}}+\int_{t-\eta}^{t}\|\mathscr{G}_{\alpha,\beta}^\nu(t-s)\left[\mathrm{B}u(s)+f(s,w(s))\right]\|_{\mathbb{W}}\mathrm{d}s\nonumber\\&\quad+\int_{0}^{t-\eta}\|\left[\mathscr{G}_{\alpha,\beta}^\nu(t-s)-\mathscr{G}_{\alpha,\beta}^\nu(\eta)\mathscr{G}_{\alpha,\beta}^\nu(t-s-\eta)\right]\left[\mathrm{B}u(s)+f(s,w(s))\right]\|_{\mathbb{W}}\mathrm{d}s\nonumber\\&\le\left\|\left[\mathscr{G}_{\alpha,\beta}^\nu(t)-\mathscr{G}_{\alpha,\beta}^\nu(\eta)\mathscr{G}_{\alpha,\beta}^\nu(t-\eta)\right]\zeta\right\|_{\mathbb{W}}+NM\eta^{1/2}\|u\|_{\mathrm{L}^{2}([t,t-\eta];\mathbb{U})}+N\int_{t-\eta}^{t}\gamma(s)\mathrm{d}s\\&\quad+\int_{0}^{t-\eta}\|\left[\mathscr{G}_{\alpha,\beta}^\nu(t-s)-\mathscr{G}_{\alpha,\beta}^\nu(\eta)\mathscr{G}_{\alpha,\beta}^\nu(t-s-\eta)\right]\left[\mathrm{B}u(s)+f(s,w(s))\right]\|_{\mathbb{W}}\mathrm{d}s\le \frac{\varepsilon}{2}.
	\end{align*}
	Therefore $$\mathfrak{A}(t)\subset \bigcup_{i=1}^{n}\mathcal{S}_{w_i}(\varepsilon ).$$
	Thus, the fact implies that the set $\mathfrak{A}(t)$ is relatively compact in $ \mathbb{W}$ for each $t\in [0,T]$.
	Consequently, invoking the Arzelà-Ascoli theorem, we deduce the compactness of the operator  $ \mathcal{P}$. 
	
    Finally, by utilizing the \emph{Schauder fixed point theorem}, we conclude that the operator $\mathcal{P}$  admits a fixed point in $Q_{r}$, which is a mild solution of the system \eqref{SEq}.
\end{proof}
We now replace the condition \textbf{\textit{(F2)}} on the nonlinear function $f(\cdot,\cdot)$ by the following weaker condition:
\begin{itemize}
    \item [\textbf{\textit{(F3)}}] For each $r>0$, there exists a positive function $\gamma_r\in\mathrm{L}^1(J;\mathbb{R}^+)$ such that 
			$$\|f(t,w)\|_{\mathbb{W}}\le\gamma_r(t) \ \mbox{a.e.}\ t\in J \ \mbox{for all}\ w\in\mathbb{W} \ \mbox{with}\ \|w\|_{\mathbb{W}}\le r.$$
\end{itemize}
\begin{theorem}\label{thm2.10}
    Assume conditions \textbf{\textit{(F1)}} and \textbf{\textit{(F3)}} are satisfied. In addition, if there exists $\tilde{r}>0$ such that
   \begin{align}\label{c1}
    N\|\zeta\|_{\mathbb{W}}+NMT^{1/2}\|u\|_{\mathrm{L}^2(J;\mathbb{U})}+N\|\gamma_{\tilde{r}}\|_{\mathrm{L}^1(J;\mathbb{R}^+)}\le \tilde{r},
   \end{align}
    then for each $u\in\mathrm{L}^{2}(J;\mathbb{U})$, the system \eqref{SEq} has at least one mild solution on $J$.
\end{theorem}
\begin{proof}
It follows immediately from condition \eqref{c1} and the proof of Step I in Theorem \ref{thm2.9} that the operator $\mathcal{P}$ maps to the set $Q_{\tilde{r}}=\{w\in C(J;\mathbb{W}): \left\|x\right\|_{C(J;\mathbb{W})}\le \tilde{r}\}$ into itself. The remaining of the proof can be completed by proceeding the same steps as in the proof of Theorem \ref{thm2.9}.
\end{proof}
\section{controllability of systems}\label{Linear}\setcounter{equation}{0}
The main objective of this section is to investigate the problem of approximate controllability of linear and semilinear control systems. For the rest of this paper, we assume that $\mathbb{W}$ is a super-reflexive Banach space (see Definition 3, \cite{JRC92}) and $\mathbb{U}$ is a separable Hilbert space (identified with its own dual), unless otherwise specified.
\subsection{Controllability operators} Let us define the following operators:
\begin{equation}\label{Copt}
	\left\{
	\begin{aligned}
		\mathrm{P}_Tu&:=\int_0^T\mathscr{G}_{\alpha,\beta}^\nu(T-s)\mathrm{B}u(t)\mathrm{d}t,\\
		\Upsilon_{0}^{T}&:=\int^{T}_{0}\mathscr{G}_{\alpha,\beta}^\nu(T-t)\mathrm{B}\mathrm{B}^{*}\mathscr{G}_{\alpha,\beta}^\nu(T-t)^{*}\mathrm{d}t=\mathrm{P}_T(\mathrm{P}_T^*),\\
		\mathrm{R}(\lambda,\Upsilon_{0}^{T})&:=(\lambda \mathrm{I}+\Upsilon_{0}^{T}\mathscr{J})^{-1},\ \lambda > 0, 
	\end{aligned}
	\right.
\end{equation}
where  $\mathrm{B}^{*}$ and $\mathscr{G}_{\alpha,\beta}^\nu(T-t)^*$ denote the adjoint operators of $\mathrm{B}$ and $\mathscr{G}_{\alpha,\beta}^\nu(T-t),$ respectively. The map $\mathscr{J}:\mathbb{W}  \rightarrow 2^{\mathbb{W}^*}$  denotes the \emph{duality mapping,} which is defined by $$\mathscr{J}[w]=\{w^*\in \mathbb{W}^*:\langle w^*, w\rangle=\|w\|_{\mathbb{W}}^{2}=\|w^*\|_{\mathbb{W}^*}^{2}\}, \ \text{ for all } \ w\in \mathbb{W}.$$ 

Now we highlight some key discussions and results related with these operators.
\begin{rem}\label{r1}
\begin{itemize}
		\item[(i)] It is evident that the operator $\Upsilon_{0}^{T}:\mathbb{W}^*\to\mathbb{W}$ is nonnegative and symmetric and the operator $\mathrm{R}(\lambda,\Upsilon_{0}^{T})$ is nonlinear. However, if $\mathbb{W}$ is a Hilbert space identified by its own dual, then the duality mapping $\mathscr{J}$ simplifies to the identity operator $\mathrm{I}$. In this case, the resolvent operator $\mathrm{R}(\lambda,\Upsilon_{0}^{T}):=(\lambda \mathrm{I}+\Upsilon_{0}^{T})^{-1},\ \lambda > 0$ is linear.  
		\item [(ii)] As $\mathbb{W}$ is a super-reflexive Banach space, then $\mathbb{W}$ can be renormed in such a way that $\mathbb{W}$ becomes uniformly smooth (cf. \cite[Proposition 5.2]{DGZ93}). Thus, the dual space $\mathbb{W}^*$ is uniformly convex by Remark \ref{r2} (see Appendix \ref{ap}). Therefore, by \cite[Proposition 32.22]{ZE90}, the mapping $\mathscr{J}$ is uniformly continuous on every bounded subsets of $\mathbb{W}$.
\item [(iii)] Since $\mathbb{W}$ is a super-reflexive Banach space, then $\mathbb{W}$ is reflexive. Then $\mathbb{W}$ can be renormed such that $\mathbb{W}$ and $\mathbb{W}^*$ become strictly convex (cf. \cite{IJM1967}). As a consequence, the mapping $\mathscr{J}$ becomes single-valued. In addition, the mapping $\mathcal{J}$  satisfies the demicontinuity property also (see \cite{Sp1978}), that is,  $$w_n\to w\ \text{ in }\ \mathbb{W}\implies \mathscr{J}[w_n] \xrightharpoonup{w} \mathscr{J}[w]\ \text{ in } \ \mathbb{W}^*\ \text{ as }\ n\to\infty.$$
	\end{itemize}
\end{rem}
If the mapping $\mathscr{J}$ is single-valued, we say that $\mathscr{J}$ is normalized duality mapping.

The following lemma is useful. 
\begin{lem}\label{lem2.16} Let $\mathbb{U}$ be a separable Hilbert space. For $\lambda>0$, the map $(\lambda\mathrm{I}+\Upsilon_{0}^{T}\mathscr{J})$ is invertible. In addition, for any given $y\in\mathbb{W}$ the following estimate
\begin{align}\label{ee1}
\left\|\lambda\mathrm{R}(\lambda,\Upsilon_{0}^{T})y]\right\|_{\mathbb{W}}\leq\left\|y\right\|_{\mathbb{W}},
	\end{align}
    holds.
\end{lem}
A proof of the above lemma can be carried out in a manner similar to the proof of \cite[Lemma 4.1]{JDE2020}.

The similar results mentioned in the above lemma can also be true if the state space $\mathbb{W}$ is a separable reflexive Banach space (see Lemma 2.2, \cite{SIAM2003}).

The following lemma implies that the operator $\mathrm{R}(\lambda,\Upsilon_{0}^{T})$ is continuous in every bounded subset of a super-reflexive Banach space $\mathbb{W}$.
\begin{lem}\label{lem2.17}
	The operator $\mathrm{R}(\lambda,\Upsilon_{0}^{T}):\mathbb{W}\to\mathbb{W},\ \lambda > 0$  is uniformly continuous in every bounded subset of $\mathbb{W}$.
\end{lem}
In view of Remark \ref{r1} and Remark \ref{r2}, it is clear that the dual space $\mathbb{W}^*$ is uniformly convex. Hence, a proof is proceeding to a similar manner of the proof of \cite[Lemma 2.24]{SAAK2024}.
\subsection{Optimal control problem and approximate controllability}\label{sublinear}
This subsection first concentrates on an optimal control problem associated with the linear control system corresponding to \eqref{SEq}. Then we discuss some important results related to the approximate controllability of the linear control problem. To do this, we consider a linear regulator problem, and the objective is to minimize a cost functional defined as follows:
\begin{equation}\label{CF}
	\mathscr{H}(w,u)=\left\|w(T)-\zeta_1\right\|^{2}_{\mathbb{W}}+\lambda\int^{T}_{0}\left\|u(t)\right\|^{2}_{\mathbb{U}}\mathrm{d}t,
\end{equation}
where $\lambda >0, \zeta_1\in\mathbb{W}$ and $w(\cdot)$ is a mild solution of the system
\begin{equation}\label{LEq2}
	\left\{
	\begin{aligned}
		w'(t)&=\mathrm{A}w(t)+\int_{0}^{t}\kappa(t-s)\mathrm{A}w(s)\mathrm{d}s+\mathrm{B}u(t),\ t\in(0,T],\\
		w(0)&=\zeta,
	\end{aligned}
	\right.
\end{equation}
with control $u\in\mathcal{U}_{\mathrm{ad}}=\mathrm{L}^2(J;\mathbb{U})$ (class of \emph{admissible controls}). It is easy to infer that $\mathrm{B}u\in\mathrm{L}^1(J;\mathbb{W})$. Then, by Proposition \ref{pre2.7}, the system \eqref{LEq2} has a unique mild solution $w\in \mathrm{C}(J;\mathbb{W}) $ given by
\begin{align*}
	w(t)= \mathscr{G}_{\alpha,\beta}^\nu(t)\zeta+\int^{t}_{0}\mathscr{G}_{\alpha,\beta}^\nu(t-s)\mathrm{B}u(s)\mathrm{d}s,\ t\in J.
\end{align*}
The \emph{admissible class} of the pair is defined as $$\mathcal{A}_{\mathrm{ad}}:=\big\{(w,u) :w\ \text{is a unique mild solution of}\ \eqref{LEq2} \ \text{associated with the control}\ u\in\mathcal{U}_{\mathrm{ad}}\big\}.$$ For any $u \in \mathcal{U}_{\mathrm{ad}}$, there exists a unique mild solution of the system \eqref{LEq2}. Consequently, the set $\mathcal{A}_{\mathrm{ad}}$ is nonempty.

In the next theorem we give the existence of an optimal pair for the minimization problem associated with the cost functional defined in \eqref{CF}. The proof of the theorem employs a technique similar to that presented in \cite{MTM-20}.
\begin{theorem}\label{exi}
	For any $\zeta\in\mathbb{W}$, the problem 
    \begin{align}\label{opt}
	\min_{ (w,u) \in \mathcal{A}_{\mathrm{ad}}}  \mathscr{H}(w,u),
\end{align}
possesses a unique optimal pair $(w^0,u^0)\in\mathcal{A}_{\mathrm{ad}}$.
\end{theorem}
\begin{proof}
     Let us define 
    $$R := \inf \limits _{u \in \mathcal{U}_{\mathrm{ad}}} \mathscr{H}(w, u).$$ 
    Since \(0 \leq R < +\infty\), there exists a minimizing sequence \(\{u^n\}_{n=1}^{\infty} \subset \mathcal{U}_{\mathrm{ad}}\) such that 
    $$\lim_{n \to \infty} \mathscr{H}(w^n, u^n) = R,$$ 
    where \((w^n, u^n) \in \mathcal{A}_{\mathrm{ad}}\) for each \(n \in \mathbb{N}\). For each \(n\), the function \(w^n(\cdot)\) is given by 
    \begin{align}\label{eq3.6}
        w^n(t) = \mathscr{G}_{\alpha, \beta}^\nu(t)\zeta + \int_0^t \mathscr{G}_{\alpha, \beta}^\nu(t-s) \mathrm{B}u^n(s) \, \mathrm{d}s, \quad t \in J.
    \end{align}
    Since \(0 \in \mathcal{U}_{\mathrm{ad}}\), we may assume without loss of generality that \(\mathscr{H}(w^n, u^n) \leq \mathscr{H}(w, 0)\), where \((w, 0) \in \mathcal{A}_{\mathrm{ad}}\). This implies
    \begin{align*}
        \left\|w^n(T) - \zeta_1\right\|_{\mathbb{W}}^2 + \lambda \int_0^T \|u^n(t)\|_{\mathbb{U}}^2 \, \mathrm{d}t 
        \leq \left\|w(T) - \zeta_1\right\|_{\mathbb{W}}^2 
        \leq 2\left(\|w(T)\|_{\mathbb{W}}^2 + \|\zeta_1\|_{\mathbb{W}}^2\right) < +\infty.
    \end{align*}
	This estimate confirms the existence of a constant $L > 0$, large enough, such that
	\begin{align}\label{e3.7}\int_0^T \|u^n(t)\|^2_{\mathbb{U}} \mathrm{d} t \leq  L< +\infty 
		.\end{align} 
        Using \eqref{eq3.6}, we compute
	\begin{align*}
\|w^n(t)\|_{\mathbb{W}}&\leq\|\mathscr{G}_{\alpha,\beta}^\nu(t)\zeta\|_{\mathbb{W}}+\int_0^t\|\mathscr{G}_{\alpha,\beta}^\nu(t-s)\mathrm{B}u^n(s)\|_{\mathbb{W}}\mathrm{d}s \nonumber\\&\leq\|\mathscr{G}_{\alpha,\beta}^\nu(t)\|_{\mathcal{L}(\mathbb{W})}\|\zeta\|_{\mathbb{W}}+\int_0^t\|\mathscr{G}_{\alpha,\beta}^\nu(t-s)\|_{\mathcal{L}(\mathbb{W})}\|\mathrm{B}\|_{\mathcal{L}(\mathbb{U};\mathbb{W})}\|u^n(s)\|_{\mathbb{U}}\mathrm{d}s\nonumber\\&\leq N\|\zeta\|_{\mathbb{W}}+NMt^{1/2}\left(\int_0^t\|u^n(s)\|_{\mathbb{U}}^2\mathrm{d}s\right)^{1/2}\nonumber\\&\leq N\|\zeta\|_{\mathbb{W}}+KMt^{1/2}L^{1/2}<+\infty, \ \mbox{for all}\ t\in J.
	\end{align*}
	 The above relation implies that the sequence $\{w^n\}_{n=1}^{\infty}$ is bounded in the space $\mathrm{L}^{2}(J;\mathbb{W})$. Then by the Banach-Alaoglu theorem, we can extract a subsequence $\{w^{n_j}\}_{j=1}^{\infty}$ of $\{w^n\}_{n=1}^{\infty}$ such that 
	\begin{align*}
		w^{n_j}\xrightharpoonup{w}w^0\ \text{ in }\mathrm{L}^{2}(J;\mathbb{W})\ \text{ as }\ j\to\infty. 
	\end{align*}
	Similarly, from the estimate \eqref{e3.7}, the sequence $\{u^n\}_{n=1}^{\infty}\subset\mathrm{L}^2(J;\mathbb{U})$ is uniformly bounded. Once again by the Banach-Alaoglu theorem, there exists a subsequence  $\{u^{n_j}\}_{j=1}^{\infty}$ of $\{u^n\}_{n=1}^{\infty}$ such that 
	\begin{align*}
		u^{n_j}\xrightharpoonup{w}u^0\ \text{ in }\mathrm{L}^2(J;\mathbb{U}) \ \text{ as } \ j\to\infty. 
	\end{align*}
    Also we have
	\begin{align}\label{e3.8}
		\mathrm{B}	u^{n_j}\xrightharpoonup{w}\mathrm{B}u^0\ \text{ in }\mathrm{L}^2(J;\mathbb{W})\ \text{ as } \ j\to\infty. 
	\end{align}
	From the weak convergence \eqref{e3.8} and the compactness of the operator $(Qg)(\cdot)= \int_{0}^{\cdot}\mathscr{G}_{\alpha,\beta}^\nu(\cdot-s)g(s)\mathrm{d}s : \mathrm{L}^{2}(J;\mathbb{W}) \rightarrow \mathrm{C}(J;\mathbb{W})$ (Lemma \ref{lem2.12}), we have that
	\begin{align*}
		&	\int_0^t\mathscr{G}_{\alpha,\beta}^\nu(t-s)\mathrm{B}u^{n_j}(s)\mathrm{d}s\to\int_0^t\mathscr{G}_{\alpha,\beta}^\nu(t-s)\mathrm{B}u^0(s)\mathrm{d}s\ \mbox{in}\ C(J;\mathbb{W}) \ \text{ as } \ j\to\infty. 
	\end{align*}
	 Moreover, we obtain
	\begin{align*}
		\sup_{t\in J}\|w^{n_j}(t)-w^0(t)\|_{\mathbb{W}}&=\sup_{t\in J}\left\|\int_0^t\mathscr{G}_{\alpha,\beta}^\nu(t-s)\mathrm{B}u^{n_j}(s)\mathrm{d}s-\int_0^t\mathscr{G}_{\alpha,\beta}^\nu(t-s)\mathrm{B}u^0(s)\mathrm{d}s\right\|_{\mathbb{W}}\nonumber\\&\to 0\ \text{ as }\ j\to\infty, 
	\end{align*}
	where 
	\begin{align*}
		w^{0}(t)=\mathscr{G}_{\alpha,\beta}^\nu(t)\zeta+\int^{t}_{0}\mathscr{G}_{\alpha,\beta}^\nu(t-s)\mathrm{B}u^0(s)\mathrm{d}s, \ t\in J.
	\end{align*}
    Thus, $w^{0}(\cdot)$ is the unique mild solution of the system \eqref{LEq2} corresponding with the control $u^0$. As $w^0(\cdot)$ is a unique mild solution of \eqref{LEq2}, therefore the whole sequence $\{w^n\}_{n=1}^{\infty}$ converges to  $w^0$. The fact that  $u^0\in\mathcal{U}_{\mathrm{ad}}$ immediately implies that $(w^0,u^0)\in\mathcal{A}_{\mathrm{ad}}$.

    Finally, we claim that  $(w^0,u^0)$ is a minimizer, that is, \emph{$R=\mathscr{H}(w^0,u^0)$}. Observe that the cost functional $\mathscr{H}(\cdot,\cdot)$, as defined in \eqref{CF}, is both convex and continuous on $\mathrm{L}^2(J;\mathbb{W}) \times \mathrm{L}^2(J;\mathbb{U})$ (cf. Proposition III.1.6 and III.1.10,  \cite{CP1983}). These properties ensure that  $\mathscr{H}(\cdot,\cdot)$ is sequentially weakly lower semi-continuous (cf. Proposition II.4.5, \cite{CP1983}). Specifically, for a sequence
	$$(w^n,u^n)\xrightharpoonup{w}(w^0,u^0)\ \text{ in }\ \mathrm{L}^2(J;\mathbb{W}) \times  \mathrm{L}^2(J;\mathbb{U}) \ \mbox{as}\ n\to\infty,$$
	we have 
	\begin{align*}
		\mathscr{H}(w^0,u^0) \leq  \liminf \limits _{n\rightarrow \infty} \mathscr{H}(w^n,u^n).
	\end{align*}
	Consequently, we deduce 
	\begin{align*}R \leq \mathscr{H}(w^0,u^0) \leq  \liminf \limits _{n\rightarrow \infty} \mathscr{H}(w^n,u^n)=  \lim \limits _{n\rightarrow \infty} \mathscr{H}(w^n,u^n) = R.\end{align*}
	Hence, $(w^0,u^0)$ is a minimizer of the problem \eqref{opt}. The uniqueness arguments readily follows from the facts the cost functional defined in \eqref{CF} is convex, the constraint \eqref{LEq2} is linear and $\mathcal{U}_{\mathrm{ad}}=\mathrm{L}^2(J;\mathbb{U})$ is convex space. Hence proof is complete.
\end{proof}
Before presenting the expression of the optimal control ${u}$, we first note that the differentiability of the mapping $p(w)=\frac{1}{2}\|w\|_{\mathbb{W}}^2$, is explicitly given by 
$$\langle\partial_w p(w),z\rangle=\frac{1}{2}\frac{\mathrm{d}}{\mathrm{d}\varepsilon}\|w+\varepsilon z\|_{\mathbb{W}}^2\Big|_{\varepsilon=0}=\langle\mathscr{J}[w],z\rangle,$$ for $z\in\mathbb{W}$, where $\partial_w p(w)$ denotes the G\^ateaux derivative of $p$ at $w\in\mathbb{W}$, (cf.  \cite{SPV2001}).

The explicit expression of the optimal control $u$ is provided in the following lemma.
\begin{lem}\label{lem3.2}
	The optimal control ${u}$ that minimize the cost functional \eqref{CF} is given by
	\begin{align*}
		{u}(t)=\mathrm{B}^{*}\mathscr{G}_{\alpha,\beta}^\nu(T-t)^*\mathscr{J}\left[\mathrm{R}(\lambda,\Upsilon_0^{T})l(w(\cdot))\right],\ t\in J,
	\end{align*}
	where
	\begin{align*}
		l(w(\cdot))=\zeta_1-\mathscr{G}_{\alpha,\beta}^\nu(T)\zeta.
	\end{align*}
\end{lem}
This lemma can be proven by adapting the technique used in \cite[Lemma 3.4]{MCRF2021}.

Subsequently, we present the approximate controllability result of the linear control system \eqref{LEq2} by using the above control. We first define the problem of approximate controllability of the control system.
\begin{Def}
    The semilinear system \eqref{SEq} (or linear system \eqref{LEq2}) is approximately controllable over $J$, if for any given initial value $\zeta\in\mathbb{W}$ and any desired state $\zeta_1\in\mathbb{W}$, and for every $\epsilon>0$, we can find a mild solution of \eqref{SEq} corresponding to a suitable control $u\in\mathrm{L}^2(J;\mathbb{U})$ such that $$\|w(T)-\zeta_1\|_{\mathbb{W}}<\epsilon.$$
\end{Def}
\begin{lem}\label{lem3.3}
	The linear control system \eqref{LEq2} is approximately controllable on $J$ if and only if 
    $$\lim_{\lambda\to 0}\lambda\mathrm{R}(\lambda,\Upsilon_{0}^{T})=0, \ \mbox{in the norm topology of}\ \mathbb{W}.$$
\end{lem}
A proof of the above lemma can be obtained through a straightforward modification of Theorem 3.2, as presented in \cite{NAHS2021}.

We now discuss some other characteristics of the approximate controllability of the linear control problem \eqref{LEq2}.
\begin{lem}\label{lem3.4}
    The system \eqref{LEq2} is approximately controllable on $J$ if and only if
$$\mathrm{B}^*\mathscr{G}_{\alpha,\beta}^\nu(T-t)^*w^*=0,\ t\in J  \implies w^*=0, \ \mbox{for}\ w\in\mathbb{W}^*.$$
\end{lem}
\begin{rem}\label{rem3.4}
	If the operator $\Upsilon_{0}^{T}$ is positive, then the linear system \eqref{LEq2}  is approximately controllable, and  the converse is also true. The positivity of $\Upsilon_{0}^{T}$ is equivalent to $$ \langle w^*, \Upsilon_{0}^{T}w^*\rangle=0\implies w^*=0.$$ Further we have 
	\begin{align*}
		\langle w^*, \Upsilon_{0}^{T}w^*\rangle =\int_{0}^T\left\|\mathrm{B}^*\mathscr{G}_{\alpha,\beta}^\nu(T-t)^*w^*\right\|_{\mathbb{U}}^2\mathrm{d}t.
	\end{align*}
	From this fact and Lemma \ref{lem3.4}, it is infer that the approximate controllability of the linear system \eqref{LEq2} is equivalent to the operator $\Upsilon_{0}^{T}$ is positive. 
\end{rem}
\subsection{Semilinear control problem} In this subsection, we demonstrate sufficient conditions of the approximate controllability for the semilinear system \eqref{SEq}. To this end, we first derive the existence of a mild solution of the system \eqref{SEq} with the control given by 
\begin{align}\label{Cot}
	u^{\lambda}(t)=u^{\lambda}(t;w)=\mathrm{B}^{*}\mathscr{G}_{\alpha,\beta}^\nu(T-t)^*\mathscr{J}\left[\mathrm{R}(\lambda,\Upsilon_{0}^{T})k(w(\cdot))\right],\ t\in J,
\end{align}
with 
\begin{align}\label{e4.2}
	k(w(\cdot))=\zeta_1-\mathscr{G}_{\alpha,\beta}^\nu(T)\zeta-\int_{0}^{T}\mathscr{G}_{\alpha,\beta}^\nu(t-s)f(s,w(s))\mathrm{d}s,
\end{align}
where $\lambda>0$, $\zeta_1\in\mathbb{W}$.
\begin{theorem}\label{thm4.1}
	If Assumptions \textbf{\textit{(F1)}}-\textbf{\textit{(F2)}} are satisfied, then for every $\lambda > 0$ and any $\zeta_1 \in \mathbb{W}$, the system \eqref{SEq} with control \eqref{Cot} admits at least one mild solution.
\end{theorem}
\begin{proof}
    For each $\lambda>0$, let define an operator $\mathcal{P}^{\lambda}:C(J;\mathbb{W})\to C(J;\mathbb{W})$ such that
		\begin{align}\label{op1}
			(\mathcal{P}^{\lambda}w)(t)=\mathscr{G}_{\alpha,\beta}^\nu(t)\zeta+\int_{0}^{t}\mathscr{G}_{\alpha,\beta}^\nu(t-s)\left[\mathrm{B}u^{\lambda}(s)+f(s,w(s))\right]\mathrm{d}s, t\in J.
		\end{align}
It is evident that the existence of a mild solution for the system \eqref{SEq} is equivalent to the operator $\mathcal{P}^\lambda$ admitting a fixed point. Similar as Theorem \ref{thm2.9}, the proof is given in several steps.  
\vskip 0.1in 
\noindent\textbf{Step (1): } For each $\lambda>0$, there exists $r_{\lambda}$ such that $\mathcal{P}^{\lambda}(Q_{r_{\lambda}})\subseteq Q_{r_{\lambda}}$, where the set $Q_{r_{\lambda}}:=\{w\in C(J;\mathbb{W}): \left\|x\right\|_{C(J;\mathbb{W})}\le r_{\lambda}\}$. Let us compute
\begin{align}\label{e3.10}
	\left\|(\mathcal{P}^{\lambda}w)(t)\right\|_\mathbb{W}&\le N\|\zeta\|_{\mathbb{W}}+N\int_{0}^{t}\|f(s,w(s)\|_{\mathbb{W}}\mathrm{d}s+\left\|\int_{0}^{t}\mathscr{G}_{\alpha,\beta}^\nu(t-s)\mathrm{B}u^{\lambda}(s)\mathrm{d}s\right\|_{\mathbb{W}}\nonumber\\&\le N\|\zeta\|_{\mathbb{W}}+N\|\gamma\|_{\mathrm{L}^1(J;\mathbb{R}^+)}+\left\|\int_{0}^{t}\mathscr{G}_{\alpha,\beta}^\nu(t-s)\mathrm{B}u^{\lambda}(s)\mathrm{d}s\right\|_{\mathbb{W}}.
\end{align}
Using the properties of duality mapping, we estimate 
\begin{align}\label{e3.11}
	\left\|\int_{0}^{t}\mathscr{G}_{\alpha,\beta}^\nu(t-s)\mathrm{B}u^{\lambda}(s)\mathrm{d}s\right\|_{\mathbb{W}}&=\left\|\int_{0}^{t}\mathscr{G}_{\alpha,\beta}^\nu(t-s)\mathrm{B}\mathrm{B}^{*}\mathscr{G}_{\alpha,\beta}^\nu(T-t)^*\mathscr{J}\left[\mathrm{R}(\lambda,\Upsilon_{0}^{T})k(w(\cdot))\right]\right\|_{\mathbb{W}}\nonumber\\&\le\frac{(NM)^{2}T}{\lambda}\left\|\mathscr{J}\left[\lambda\mathrm{R}(\lambda,\Upsilon_{0}^{T})k(w(\cdot))\right]\right\|_{\mathbb{W}^{*}}\nonumber\\&= \frac{(NM)^{2}T}{\lambda}\left\|\lambda\mathrm{R}(\lambda,\Upsilon_{0}^{T})k(w(\cdot))\right\|_{\mathbb{W}}\nonumber\\&\le \frac{(NM)^{2}T}{\lambda}\|k(w(\cdot))\|_{\mathbb{W}}\nonumber\\&\le \frac{(NM)^{2}T}{\lambda}\left[\|\zeta_1\|_{\mathbb{W}}+N\|\zeta\|_{\mathbb{W}}+N\|\gamma\|_{\mathrm{L}^1(J;\mathbb{R}^+)}\right].
\end{align}
 Combining the estimates \eqref{e3.10} and \eqref{e3.11}, we obtain
\begin{align*}
	\left\|(\mathcal{P}^{\lambda}w)(t)\right\|_\mathbb{W}&\le N\|\zeta\|_{\mathbb{W}}+N\|\gamma\|_{\mathrm{L}^1(J;\mathbb{R}^+)}+ \frac{(NM)^{2}T}{\lambda}\left[\|\zeta_1\|_{\mathbb{W}}+N\|\zeta\|_{\mathbb{W}}+N\|\gamma\|_{\mathrm{L}^1(J;\mathbb{R}^+)}\right]=r_{\lambda}.
\end{align*}
From this we can conclude that for each $\lambda>0$, there is an  $r_{\lambda}>0$ such that $ \mathcal{P}^{\lambda}(Q_{r_{\lambda}})\subseteq Q_{r_{\lambda}}$.
\vskip 0.1in 
\noindent\textbf{Step (2): } We claim that the map $ \mathcal{P}^{\lambda}$ is continuous. To demonstrate this claim, let us consider a sequence $\{{w}^n\}^\infty_{n=1}\subseteq Q_r$ such that ${w}^n\rightarrow {w}\ \mbox{in}\ Q_r,$ that is,
$$\lim\limits_{n\rightarrow \infty}\left\|w^n-w\right\|_{C(J;\mathbb{W})}=0.$$
Using Assumption \textbf{\textit{(F1)}}-\textbf{\textit{(F2)}} together with Lebesgue's dominate convergence theorem,  we deduce that
\begin{align*}
	\left\|k(w^{n}(\cdot))-k(w(\cdot))\right\|_{\mathbb{W}}&\le \left\|\int^{T}_{0}\mathscr{G}_{\alpha,\beta}^\nu(T-s)\left[f(s, w^n(s))-f(s,w(s))\right]\mathrm{d}s\right\|_{\mathbb{W}}\nonumber\\&\le N\int^{T}_{0}\left\|f(s, w^n(s))-f(s,w(s))\right\|_{\mathbb{W}}\mathrm{d}s\to 0 \ \mbox{as}\ n\to\infty.
\end{align*}
Using the uniform continuous property of the mapping $\mathrm{R}(\lambda,\Upsilon_{0}^{T}):\mathbb{W}\to\mathbb{W}$ on every bounded subset of $\mathbb{W}$ (see Lemma \ref{lem2.17}), we arrive
$$\mathrm{R}(\lambda,\Upsilon_{0}^{T})k(w^{n}(\cdot))\to\mathrm{R}(\lambda,\Upsilon_{0}^{T})k(w(\cdot))\ \mbox{in}\  \mathbb{W} \ \mbox{as}\ n\to\infty.$$ Since the duality mapping $\mathscr{J}:\mathbb{W}\to\mathbb{W}^{*}$  is demicontinuous, then we have
\begin{align*}
	\mathscr{J}\left[\mathrm{R}(\lambda,\Upsilon_{0}^{T})k(w^{n}(\cdot))\right]\xrightharpoonup{w}\mathscr{J}\left[\mathrm{R}(\lambda,\Upsilon_{0}^{T})k(w(\cdot))\right]\ \text{ in }\ \mathbb{W}^{*} \ \text{as} \ n\to\infty.
\end{align*}
	Since the operator $\mathscr{G}_{\alpha,\beta}^\nu(t)$ is compact for each $t > 0$, it follows that the adjoint operator $\mathscr{G}_{\alpha,\beta}^\nu(t)^*$ is also compact for each $t > 0$. Consequently, by combining the weak convergence with the compactness of the operator $\mathscr{G}_{\alpha,\beta}^\nu(\cdot)^*$, we can readily conclude
\begin{align}\label{e4.4}
	\left\|\mathscr{G}_{\alpha,\beta}^\nu(T-t)^*\mathscr{J}\left[\mathrm{R}(\lambda,\Upsilon_{0}^{T})k(w^n(\cdot))\right]\!\!-\!\mathscr{G}_{\alpha,\beta}^\nu(T-t)^*\mathscr{J}\left[\mathrm{R}(\lambda,\Upsilon_{0}^{T})k(w(\cdot))\right]\right\|_{\mathbb{W^*}}\!\!\!\!\!\to 0 \ \text{as} \ n\to\infty,
\end{align}
for all $t\in [0,T)$. Using \eqref{Cot} and \eqref{e4.4}, it follows
\begin{align}\label{4.4.9}
	&\left\|(u^{\lambda})^{n}(t)-u^{\lambda}(t)\right\|_{\mathbb{U}}\to 0 \ \text{ as } \ n\to\infty, \text{ uniformly for all }\ t\in [0,T).
\end{align}
By \eqref{4.4.9}, Assumption  \textbf{\textit{(F1)}}-\textbf{\textit{(F2)}} and Lebesgue's dominate convergence theorem, we obtain 
\begin{align*}
	&\left\|(\mathcal{P}^{\lambda}w^{n})(t)-(\mathcal{P}^{\lambda}w)(t)\right\|_{\mathbb{W}}\nonumber\\&\leq\left\|\int_{0}^{t}\mathscr{G}_{\alpha,\beta}^\nu(t-s)\mathrm{B}\left[(u^{\lambda})^n(s)-u^{\lambda}(s)\right]\mathrm{d}s\right\|_{\mathbb{W}}+\left\|\int_{0}^{t}\mathscr{G}_{\alpha,\beta}^\nu(t-s)\left[f(s, w^n(s))-f(s,w(s))\right]\mathrm{d}s\right\|_{\mathbb{W}}\nonumber\\&\leq NMT\esssup_{t\in J}\left\|(u^{\lambda})^n(t)-u^{\lambda}(t)\right\|_{\mathbb{U}} +N\int_{0}^{t}\left\|\left[f(s,w^n(s))-f(s,w(s))\right]\right\|_{\mathbb{W}}\mathrm{d}s\nonumber\\&\to 0 \ \text{ as }\ n\to\infty,
\end{align*}
for each $t\in J$. Therefore, the map $\mathcal{P}^{\lambda}$ is continuous.
\vskip 0.1in 
\noindent\textbf{Step (3): } We now demonstrate that the operator $ \mathcal{P}^{\lambda}$ is compact for $\lambda>0$. By applying the estimate
\begin{align*}
    \|u^{\lambda}(t)\|_{\mathbb{U}}\le \frac{NM}{\lambda}\left[\|\zeta_1\|_{\mathbb{W}}+N\|\zeta\|_{\mathbb{W}}+N\|\gamma\|_{\mathrm{L}^1(J;\mathbb{R}^+)}\right],
\end{align*}
for all $t\in J$ and following the approach outlined in Step (3) of the proof of Theorem \ref{thm2.9} with the necessary modification, we can establish that $ \mathcal{P}^{\lambda}$ is compact.

Further, by applying the \emph{Schauder fixed point theorem}, we deduce that the operator $\mathcal{P}^{\lambda}$ has a fixed point in $Q_{r(\lambda)}$ for every $\lambda > 0$. Hence, the proof is completed. 
\end{proof}
We now turn our attention to investigating the approximate controllability of the semilinear system \eqref{SEq}, as presented in the theorem below.
\begin{theorem}\label{thm3.11}
  If Assumption \textbf{\textit{(F1)}}-\textbf{\textit{(F2)}} are verified and the linear system \eqref{LEq2} is approximately controllable, then the system \eqref{SEq} is approximately controllable. 
\end{theorem}
\begin{proof}
    From Theorem \ref{thm4.1}, it follows that system \eqref{SEq} has a mild solution, say, $w^{\lambda}\in Q_{r(\lambda)}$, for each $\lambda>0$ and $\zeta_1\in\mathbb{W}$, which is given by 
    \begin{align*}
        w^{\lambda}(t)=\mathscr{G}_{\alpha,\beta}^\nu(t)\zeta+\int_{0}^{t}\mathscr{G}_{\alpha,\beta}^\nu(t-s)\left[\mathrm{B}u^{\lambda}(s)+f(s, w^{\lambda}(s))\right]\mathrm{d}s,
    \end{align*}
    for each $t\in J$, with the control 
    \begin{align*}
	u^{\lambda}(t)=\mathrm{B}^{*}\mathscr{G}_{\alpha,\beta}^\nu(T-t)^*\mathscr{J}\left[\mathrm{R}(\lambda,\Upsilon_{0}^{T})k( w^{\lambda}(\cdot))\right],\ t\in J,
\end{align*}
where
\begin{align*}
	k( w^{\lambda}(\cdot))=\zeta_1-\mathscr{G}_{\alpha,\beta}^\nu(T)\zeta-\int_{0}^{T}\mathscr{G}_{\alpha,\beta}^\nu(t-s)f(s, w^{\lambda}(s))\mathrm{d}s.
\end{align*}
Let us compute 
	\begin{align}\label{4.35}
		w^{\lambda}(T)&=\mathscr{G}_{\alpha,\beta}^\nu(T)\zeta+\int_{0}^{T}\mathscr{G}_{\alpha,\beta}^\nu(T-s)\left[\mathrm{B}u^{\lambda}(s)+f(s,w^{\lambda}(s))\right]\mathrm{d}s\nonumber\\&=\mathscr{G}_{\alpha,\beta}^\nu(T)\zeta+\int_{0}^{T}\mathscr{G}_{\alpha,\beta}^\nu(T-s)f(s,w^{\lambda}(s))\mathrm{d}s\nonumber\\&\quad+\int_{0}^{T}\mathscr{G}_{\alpha,\beta}^\nu(T-s)\mathrm{B}\mathrm{B}^*\mathscr{G}_{\alpha,\beta}^\nu(T-s)^*\mathscr{J}\left[\mathrm{R}(\lambda,\Upsilon_{0}^{T})k(w^\lambda(\cdot))\right]\mathrm{d}s\nonumber\\&=\mathscr{G}_{\alpha,\beta}^\nu(T)\zeta+\int_{0}^{T}\mathscr{G}_{\alpha,\beta}^\nu(T-s)f(s,w^{\lambda}(s))\mathrm{d}s+\Upsilon_{0}^{T}\mathscr{J}\left[\mathrm{R}(\lambda,\Upsilon_{0}^{T})k(w^\lambda(\cdot))\right]\nonumber\\&=\zeta_1-\lambda\mathrm{R}(\lambda,\Upsilon_{0}^{T})k(w^\lambda(\cdot)).
	\end{align}	
    By Assumptions  \textbf{\textit{(F1)}}-\textbf{\textit{(F2)}}, we have
	\begin{align}
		\int_{0}^{T}\left\|f(s,w^{\lambda_j}(s))\right\|_{\mathbb{W}}\mathrm{d}s&\le \int_{0}^{T}\gamma(s)\mathrm{d} s<+\infty,\quad j\in\mathbb{N}. \nonumber
	\end{align}
	This fact guarantees that the sequence $\{f(\cdot, w^{\lambda_{j}}(\cdot))\}_{j=1}^{\infty}$ is uniformly integrable. Consequently, by using the Dunford-Pettis theorem, we extract a subsequence of $ \{f(\cdot, w^{\lambda_{j}}(\cdot))\}_{j=1}^{\infty}$, which we denote again as $ \{f(\cdot, w^{\lambda_{j}}(\cdot))\}_{j=1}^{\infty}$ such that
	\begin{align}\label{wc}
		f(\cdot,w^{\lambda_{j}}(\cdot))\xrightharpoonup{w}f(\cdot) \ \mbox{in}\ \mathrm{L}^1(J;\mathbb{W}) \ \mbox{as}\  \lambda_j\to 0^+ \ (j\to\infty).
	\end{align}
	Next, we evaluate
	\begin{align}\label{e4}
		&\left\|k(w^{\lambda_{j}}(\cdot))-q\right\|_{\mathbb{W}}\nonumber\\&\le\left\|\int_{0}^{T}\mathscr{G}_{\alpha,\beta}^\nu(T-s)\left[f(s,w^{\lambda_j}(s))-f(s)\right]\mathrm{d}s\right\|_{\mathbb{W}}\to 0\ \mbox{as}\ \lambda_j\to0^+ \ (j\to\infty), 
	\end{align}
	where 
	\begin{align*}
		q =\zeta_1-\mathscr{G}_{\alpha,\beta}^\nu(T)\zeta-\int_{0}^{T}\mathscr{G}_{\alpha,\beta}^\nu(T-s)f(s)\mathrm{d}s.
	\end{align*}
	The estimates \eqref{e4} goes to zero using the above weak convergences together with Corollary \ref{cor1}. 

    Let $z^{\lambda_j}=w^{\lambda_j}(T)-\zeta_{1}$ for each $j\in\mathbb{N}$ and 
    $$\|z^{\lambda_j}\|_{\mathbb{W}}\le \|k(w^{\lambda_{j}})\|_{\mathbb{W}}\le \|\zeta_1\|_{\mathbb{W}}+N\|\zeta\|_{\mathbb{W}}+N\|\gamma\|_{\mathrm{L}^1(J;\mathbb{R}^+)}<\infty.$$
    This fact implies that the sequence $\{z^{\lambda_j}\}_{j=1}^{\infty}$ is uniformly bounded in $\mathbb{W}$. Since the space $\mathbb{W}$ is reflexive, by the Banach-Alaoglu theorem, there exists a subsequence, still denoted by $\{z^{\lambda_j}\}_{j=1}^{\infty},$ such that
    $$\mathscr{J}z^{\lambda_j}\xrightharpoonup{w}q^* \ \mbox{in}\ \mathbb{W}^* \  \mbox{as}\ \lambda_j\to 0^+ (j\to \infty).$$
    Consequently, we have $$\Upsilon_{0}^{T}\mathscr{J}w^{\lambda_j}\to\Upsilon_{0}^{T}q^*\ \mbox{in}\ \mathbb{W} \ \  \mbox{as}\ \lambda_j\to 0^+ (j\to \infty).$$
    Using the equality \eqref{4.35} and the above convergence, we obtain that $\Upsilon_{0}^{T}q^*=0$. Further, using Remark \ref{rem3.4}, we get $q^*=0$. Once again by equality \eqref{4.35}, we arrive 
    $$\lambda_j\|z^{\lambda_j}\|^{2}+\langle \Upsilon_{0}^{T}\mathscr{J}z^{\lambda_j}, \mathscr{J}z^{\lambda_j} \rangle=-\lambda_j\langle k(w^{\lambda_j}(\cdot)), \mathscr{J}z^{\lambda_j}\rangle,$$
    which implies that
    $$\|z^{\lambda_j}\|_{\mathbb{W}}\le|\langle k(w^{\lambda_j}(\cdot)), \mathscr{J}z^{\lambda_j} \rangle|\le|\langle k(w^{\lambda_j}(\cdot))-q, \mathscr{J}z^{\lambda_j}\rangle|+|\langle q, \mathscr{J}z^{\lambda_j} \rangle|\to 0, j\to\infty, $$
    and $z^{\lambda_j}\to 0$ as $j\to\infty$. Hence the proof is complete.
    \end{proof}
    \begin{theorem}\label{thm4.2}
	If Assumptions \textbf{\textit{(F1)}} and \textbf{\textit{(F3)}} are satisfied, then for every $\lambda > 0$ and any $\zeta_1 \in \mathbb{W}$, the system \eqref{SEq} with the control \eqref{Cot} admits at least one mild solution. Provided, there exists $\tilde{r}_{\lambda}>0$ such that 
    \begin{align}\label{c2}
        N\|\zeta\|_{\mathbb{W}}+N\|\gamma_{\tilde{r}_{\lambda}}\|_{\mathrm{L}^1(J;\mathbb{R}^+)}+ \frac{(NM)^{2}T}{\lambda}\left[\|\zeta_1\|_{\mathbb{W}}+N\|\zeta\|_{\mathbb{W}}+N\|\gamma_{\tilde{r}_{\lambda}}\|_{\mathrm{L}^1(J;\mathbb{R}^+)}\right]\le \tilde{r}_{\lambda}.
    \end{align}
\end{theorem}
\begin{proof}
    It directly follows from condition \eqref{c2} and the argument in Step I of Theorem \ref{thm4.1} that the operator $\mathcal{P}^{\lambda}$ maps to the set $Q_{\tilde{r_{\lambda}}}=\{w\in C(J;\mathbb{W}): \left\|x\right\|_{C(J;\mathbb{W})}\le \tilde{r_{\lambda}}\}$ into itself. The rest of the proof proceeds analogously to the steps in the proof of Theorem \ref{thm4.1}.
\end{proof}
\begin{rem}
    Note that if we replace Assumption \textbf{\textit{(F2)}} by \textbf{\textit{(F3)}} in Theorem \ref{thm3.11}, then by the above theorem for each $\lambda>0$, there exists a mild solution, say, $w^{\lambda}(\cdot)$, corresponding to the control given in \eqref{Cot}. Subsequently, we have
    \begin{align*}
        \|w^{\lambda}(t)\|_{\mathbb{W}}\le  N\|\zeta\|_{\mathbb{W}}+N\|\gamma_{r}\|_{\mathrm{L}^1(J;\mathbb{R}^+)}+ \frac{(NM)^{2}T}{\lambda}\left[\|\zeta_1\|_{\mathbb{W}}+N\|\zeta\|_{\mathbb{W}}+N\|\gamma_{r}\|_{\mathrm{L}^1(J;\mathbb{R}^+)}\right]
    \end{align*}
    Due to the $\frac{1}{\lambda}$-term in the right-hand side of the above expression, the sequence $\{f(\cdot, w^{\lambda_{j}}(\cdot))\}_{j=1}^{\infty}$ is fail to be uniformly integrable, consequently, we cannot conclude the weak convergence estimate given in \eqref{wc}.
    \end{rem}
    In the process, we now present another result of the approximate controllability of the system \eqref{SEq} under assumptions \textbf{\textit{(F1)}} and \textbf{\textit{(F3)}} on the nonlinear term $f(\cdot,\cdot)$. We also assume that the following
    \begin{itemize}
    \item [\textbf{\textit{(F)}}] There exists a constant $\tilde{L}\ge 1$ such that 
            $$\|F(t)F(T)^{-1}\|_{\mathcal{L}(\mathcal{R}(F(T)),\mathbb{W})}\le\tilde{L}.$$
            Here, for each $t\in J$, the operator $F(t):\mathbb{W}^*\to\mathbb{W}$ is defined as
            \begin{align}\label{op11}
                F(t)w^*=\int^{t}_{0}\mathscr{G}_{\alpha,\beta}^\nu(t-s)\mathrm{B}\mathrm{B}^{*}\mathscr{G}_{\alpha,\beta}^\nu(T-s)^{*}w^*\mathrm{d}s,
            \end{align}
            and $\mathcal{R}(F(T))$ represents the range of $F(T)$. Note that $F(T)=\Upsilon_{0}^{T}$.
    \end{itemize}
    Note that if the linear control system is approximately controllable in $J$, then by Remark 4.2, \cite{JDE2020}, the operator $F(T)$ is injective. Consequently, the operators $F(T)^{-1}:\mathcal{R}(F(T))\to\mathbb{W}^*$ and $F(t)F(T)^{-1}:\mathcal{R}(F(T))\to\mathbb{W}$ are well defined.
     \begin{theorem}\label{thm4.4}
	 Assume that \textbf{\textit{(F1)}}, \textbf{\textit{(F3)}} and  \textbf{\textit{(F)}} are satisfied and the linear system \eqref{LEq2} is approximately controllable. If there exists $\tilde{r}>0$ such that 
    \begin{align}\label{cc4}
        N\|\zeta\|_{\mathbb{W}}+N\|\gamma_{\tilde{r}}\|_{\mathrm{L}^1(J;\mathbb{R}^+)}+ 2\tilde{L}\left[\|\zeta_1\|_{\mathbb{W}}+N\|\zeta\|_{\mathbb{W}}+N\|\gamma_{\tilde{r}}\|_{\mathrm{L}^1(J;\mathbb{R}^+)}\right]\le \tilde{r},
    \end{align}
    then the system \eqref{SEq} is approximately controllable.
    \end{theorem}
    \begin{proof}
         For each $\lambda>0$, we keep the operator $\mathcal{P}^{\lambda}:C(J;\mathbb{W})\to C(J;\mathbb{W})$ is same as in the proof of Theorem \ref{thm4.1}, that is,
		\begin{align*}
        (\mathcal{P}^{\lambda}w)(t)=\mathscr{G}_{\alpha,\beta}^\nu(t)\zeta+\int_{0}^{t}\mathscr{G}_{\alpha,\beta}^\nu(t-s)\left[\mathrm{B}u^{\lambda}(s)+f(s,w(s))\right]\mathrm{d}s, t\in J.
		\end{align*}
We now prove that, for all $\lambda>0$, there is an $\tilde{r}$ such that $\mathcal{P}^{\lambda}(Q_{\tilde{r}})\subseteq Q_{\tilde{r}}$, where the set $Q_{\tilde{r}}:=\{w\in C(J;\mathbb{W}): \left\|x\right\|_{C(J;\mathbb{W})}\le \tilde{r}\}$. Let us estimate
\begin{align}\label{ee3.10}
	\left\|(\mathcal{P}^{\lambda}w)(t)\right\|_\mathbb{W}&\le N\|\zeta\|_{\mathbb{W}}+N\int_{0}^{t}\|f(s,w(s)\|_{\mathbb{W}}\mathrm{d}s+\left\|\int_{0}^{t}\mathscr{G}_{\alpha,\beta}^\nu(t-s)\mathrm{B}u^{\lambda}(s)\mathrm{d}s\right\|_{\mathbb{W}}\nonumber\\&\le N\|\zeta\|_{\mathbb{W}}+N\|\gamma_{\tilde{r}}\|_{\mathrm{L}^1(J;\mathbb{R}^+)}+\left\|\int_{0}^{t}\mathscr{G}_{\alpha,\beta}^\nu(t-s)\mathrm{B}u^{\lambda}(s)\mathrm{d}s\right\|_{\mathbb{W}}.
\end{align}
Using Assumption \textbf{\textit{(F)}} along with the estimate \eqref{ee1}, we obtain 
\begin{align}\label{ee3.11}
	\left\|\int_{0}^{t}\mathscr{G}_{\alpha,\beta}^\nu(t-s)\mathrm{B}u^{\lambda}(s)\mathrm{d}s\right\|_{\mathbb{W}}&=\left\|\int_{0}^{t}\mathscr{G}_{\alpha,\beta}^\nu(t-s)\mathrm{B}\mathrm{B}^{*}\mathscr{G}_{\alpha,\beta}^\nu(T-t)^*\mathscr{J}\left[\mathrm{R}(\lambda,\Upsilon_{0}^{T})k(w(\cdot))\right]\right\|_{\mathbb{W}}\nonumber\\&=\left\|F(t)\mathscr{J}\left[\mathrm{R}(\lambda,\Upsilon_{0}^{T})k(w(\cdot))\right]\right\|_{\mathbb{W}}\nonumber\\&=\left\|F(t)F(T)^{-1}\Upsilon_{0}^{T}\mathscr{J}\left[\mathrm{R}(\lambda,\Upsilon_{0}^{T})k(w(\cdot))\right]\right\|_{\mathbb{W}}\nonumber\\&\le\tilde{L}\left\|\Upsilon_{0}^{T}\mathscr{J}\left[\mathrm{R}(\lambda,\Upsilon_{0}^{T})k(w(\cdot))\right]\right\|_{\mathbb{W}}\nonumber\\&=\tilde{L}\left\|(\lambda\mathrm{I}+\Upsilon_{0}^{T}\mathscr{J}-\lambda\mathrm{I})\mathrm{R}(\lambda,\Upsilon_{0}^{T})k(w(\cdot))\right\|_{\mathbb{W}}\nonumber\\&\le 2\tilde{L}\left\|k(w(\cdot))\right\|_{\mathbb{W}} \nonumber\\&\le 2\tilde{L}\left[\|\zeta_1\|_{\mathbb{W}}+N\|\zeta\|_{\mathbb{W}}+N\|\gamma_{\tilde{r}}\|_{\mathrm{L}^1(J;\mathbb{R}^+)}\right].
    \end{align}
 Combining the estimates \eqref{ee3.10} and \eqref{ee3.11}, we have
\begin{align}\label{ee3}
	\left\|(\mathcal{P}^{\lambda}w)(t)\right\|_\mathbb{W}&\le N\|\zeta\|_{\mathbb{W}}+N\|\gamma_{\tilde{r}}\|_{\mathrm{L}^1(J;\mathbb{R}^+)}+ 2\tilde{L}\left[\|\zeta_1\|_{\mathbb{W}}+N\|\zeta\|_{\mathbb{W}}+N\|\gamma_{\tilde{r}}\|_{\mathrm{L}^1(J;\mathbb{R}^+)}\right].
\end{align}
By the above inequality and the estimate \eqref{cc4}, we deduce that, for all $\lambda>0$, there exists $\tilde{r}$ (independent from $\lambda$) such that $\mathcal{P}^{\lambda}(Q_{\tilde{r}})\subseteq Q_{\tilde{r}}$. 

The remaining proof of the existence of a mild solution for the system \eqref{SEq} proceeds analogously to the steps in the proof of Theorem \ref{thm4.1}. Moreover, Assumptions \textbf{\textit{(F1)}} and \textbf{\textit{(F3)}}, along with the fact \eqref{ee3}, ensure that the sequence $\{f(\cdot, w^{\lambda_{j}}(\cdot))\}_{j=1}^{\infty}$ is uniformly integrable. Consequently, by following a similar approach to the proof of Theorem \ref{thm3.11}, we can show that the system \eqref{SEq} is approximately controllable
    \end{proof}
    We now discuss the approximate controllability of the system \eqref{SEq} within the framework of a general Banach space $\mathbb{W}$.
\begin{theorem}\label{thm3.15}
		If Assumptions \textbf{\textit{(F1)}}-\textbf{\textit{(F2)}} (or \textbf{\textit{(F3)}}) are satisfied and the linear system \eqref{LEq2} is approximately controllable on $[0,t]$ for all $0<t\le T$, then the system \eqref{SEq} is approximately controllable.
	\end{theorem}
    The theorem can be proved by adapting the technique used in \cite[Theorem 4.18]{EECT2025}.

    \section{Application}\label{sec5}\setcounter{equation}{0}
	In this section, we investigate the approximate controllability of a heat equation with memory that demonstrates the applicability of our developed results. The analysis is carried out within the functional framework of the state space $\mathrm{L}^p([0,\pi];\mathbb{R})$ for $p\in[2,\infty)$, and the control space $\mathrm{L}^2([0,\pi];\mathbb{R}),$ as discussed in \cite{MTM-20}.
	\begin{Ex}\label{ex1} Consider the following control system:
		\begin{equation}\label{57}
			\left\{
			\begin{aligned}
				\frac{\partial}{\partial t}w(t,\xi)&=\frac{\partial^2w(t,\xi)}{\partial \xi^2}+ \frac{\alpha}{\Gamma(\nu)}\int_{0}^{t}e^{-\beta(t-s)}(t-s)^{\nu-1}\frac{\partial^2w(s,\xi)}{\partial \xi^2}\mathrm{d}s+\eta(t,\xi)\\&\quad+h(t,w(t,\xi)), \ t\in (0,T], \xi\in[0,\pi], \\
				&w(t,0)=w(t,\pi)=0, \  t\in J=[0,T],\\
				&w(0,\xi)=w_{0}(\xi),\ \xi\in[0,\pi],
			\end{aligned}
			\right.
		\end{equation}
		where $\alpha>0, \beta\ge 0, 0<\nu<1$ and $w_0:[0,\pi]\to\mathbb{R}$ is an appropriate function. The function $\eta: J\times[0,\pi]\to\mathbb{R}$ is square integrable in $t$ and $\xi$. 
	\end{Ex}
    \noindent\textbf{Case 1:} Take the state space $\mathbb{W}_p=\mathrm{L}^p([0,\pi];\mathbb{R})$ for $p\in[2,\infty)$ and the control space $\mathbb{U}=\mathrm{L}^2([0,\pi];\mathbb{R})$. Note that the space $\mathbb{W}_p$ for $p\in[2,\infty)$ is super-reflexive and its dual space $\mathbb{W}_p^*=\mathrm{L}^{\frac{p}{p-1}}([0,\pi];\mathbb{R})$ is uniform convex.
	\vskip 0.1 cm
	\noindent\textbf{Step 1:} \emph{Resolvent operator.} We now define the operator
	$\mathrm{A}_{p}:D(\mathrm{A}_p)\subset \mathbb{W}_p\to\mathbb{W}_p$ as
	$$\mathrm{A}_pg=g'',\ \ D(\mathrm{A}_p)=\mathrm{W}^{2,p}([0,\pi];\mathbb{R})\cap\mathrm{W}^{1,p}_0([0,\pi];\mathbb{R}).$$
	The spectrum of the operator is given by $\{-m^2: m\in\mathbb{N}\}$. Thus, the operator $\mathrm{A}_p$ can be written as 
	\begin{align}\label{OPA}
		\mathrm{A}_pg=\sum_{m=1}^{\infty}-m^2\langle g, \nu_m \rangle \nu_m,\ g\in D(\mathrm{A}_p),
	\end{align}
	where $ \nu_m(\xi)=\sqrt{\frac{2}{\pi}}\sin(m\xi)$ and $\langle g, v_m \rangle:=\int_{0}^{\pi}g(\xi)v_m(\xi)\mathrm{d}\xi$. The domain $D(\mathrm{A}_p)$ is dense in $\mathbb{W}_p$ because $\mathrm{C}_0^{\infty}([0,\pi];\mathbb{R})\subset D(\mathrm{A}_p)$. Next, for all $\lambda\neq -m^2, m\in\mathbb{N}$, let us consider the following Sturm-Liouville equation :
\begin{equation*}
\left\{
\begin{aligned}
\left(\lambda\mathrm{I}-\mathrm{A}_p\right)g(\xi)&=h(\xi), \ 0<\xi<\pi,\\
g(0)=g(\pi)&=0.
\end{aligned}
\right.
\end{equation*}
The above equation can be written as
\begin{align}\label{511}
\left(\lambda\mathrm{I}-\Delta\right)g(\xi)&=h(\xi),
\end{align}
where $\Delta g(\xi)=g''(\xi)$. Multiplying $g|g|^{p-2}$ both sides in equation \eqref{511} and then integrating over the interval $[0,\pi]$, we obtain
\begin{align*}
\lambda\int_0^{\pi}|g(\xi)|^p\mathrm{d}\xi&+(p-1)\int_0^{\pi}|g(\xi)|^{p-2}|g'(\xi)|^2\mathrm{d}\xi=\int_0^{\pi}h(\xi)g(\xi)|g(\xi)|^{p-2}\mathrm{d}\xi.
\end{align*}
Applying H\"older's inequality, we obtain
\begin{align*}
\lambda\int_0^{\pi}|g(\xi)|^p\mathrm{d}\xi&\leq \left(\int_0^{\pi}|g(\xi)|^p\mathrm{d}\xi\right)^{\frac{p-1}{p}}\left(\int_0^{\pi}|h(\xi)|^p\mathrm{d}\xi\right)^{\frac{1}{p}}.
\end{align*}
Hence, we have
\begin{align*}
\|\mathrm{R}(\lambda,\mathrm{A}_p)h\|_{\mathrm{L}^p}=\|g\|_{\mathrm{L}^p}\leq\frac{1}{\lambda}\|h\|_{\mathrm{L}^p},
\end{align*}
so that we get 
\begin{align*}
\|\mathrm{R}(\lambda,\mathrm{A}_p)\|_{\mathcal{L}(\mathrm{L}^p)}\leq\frac{1}{\lambda}, \ \mbox{for all}\ \lambda>0.
\end{align*}
For complex values of $\lambda$, similar estimates can be established by multiplying \eqref{511} by $\overline{g}|g|^{p-2}$, performing integration by parts over $(0,\pi)$,  separating the real and imaginary components, and subsequently combining them. Thus, the operator $\mathrm{A}_p$ is the infinitesimal generator of strongly continuously semigroup. Also the operator $\mathrm{A}_p$ is sectorial of angle $\theta$ for every $0<\theta<\frac{\pi}{2}$ (see \cite{RP2021}). Therefore, the abstract form of the linear system corresponding to \eqref{57} has a resolvent family $\mathscr{G}_{\alpha,\beta,p}^\nu(\cdot)$ on $\mathbb{W}_p$. Moreover, the resolvent operator $\mathrm{R}(\lambda,\mathrm{A}_p)$ is compact for some $\lambda\in\rho(\mathrm{A}_p)$ (see application section, \cite{MTM-20}). Hence, the resolvent family $\mathscr{G}_{\alpha,\beta,p}^\nu(\cdot)$ is compact for all $t>0$.

Since $\mathrm{A}_p$ is a  sectorial operator, by \cite[Corollary 5.4]{RP2021}, the resolvent family $\mathscr{G}_{\alpha,\beta,p}^\nu(t)$ can be given explicitly by
\begin{equation*}
    \mathscr{G}_{\alpha,\beta,p}^\nu(t)=e^{-\beta t}\sum_{m=1}^\infty\sum_{k=0}^\infty(\beta-\lambda_m)^kt^k E_{\nu+1,k+1}^{k+1}(-\alpha\lambda_m t^{\nu+1}),\;\;t\ge 0,
\end{equation*}
with $\lambda_m=m^2,$ where for any $q,r,\gamma>0$
\begin{align*}
E_{q,r}^{\gamma}(z):=\sum_{j=0}^\infty \frac{\Gamma(\gamma+j)z^j}{j!\Gamma(\gamma)\Gamma(qj+r)}, \;\;\;z\in\mathbb{C},
\end{align*}
denotes the generalized Mittag-Leffler function. Since
\begin{equation*}
    \frac{d^n }{dz^n}[z^{r-1}E_{q,r}^{\gamma}(az^q)]=z^{r-n-1}E_{q,r-n}^{\gamma}(az^q),
\end{equation*}
(see for instance \cite[Theorem 11.1]{Ha-Ma-Sa-11}). By the general Leibniz rule for the $n$-derivative, we compute
\begin{align*}
    \frac{d^n \mathscr{G}_{\alpha,\beta,p}^\nu(t)}{dt^n}=e^{-\beta t}\sum_{m=1}^\infty \sum_{j=0}^n \sum_{k=j}^\infty {n\choose j}(-\beta)^{n-j} (\beta-\lambda_m)^kt^{k-j} E_{\nu+1,k+1-j}^{k+1}(-\alpha\lambda_m t^{\nu+1}),\;\;t\ge 0.
\end{align*}
Then, for any $g\in D(\mathrm{A}_p)$, we get 
\begin{align*}
    \frac{d^n \mathscr{G}_{\alpha,\beta,p}^\nu(t)}{dt^n}g=e^{-\beta t}\sum_{m=1}^\infty \sum_{j=0}^n \sum_{k=j}^\infty {n\choose j}(-\beta)^{n-j} (\beta-\lambda_m)^kt^{k-j} E_{\nu+1,k+1-j}^{k+1}(-\alpha\lambda_m t^{\nu+1})\langle g, \nu_m \rangle \nu_m,
\end{align*}
for all $t\ge 0.$ This implies that
\begin{align*}
    \frac{d^n \mathscr{G}_{\alpha,\beta,p}^\nu(0)}{dt^n}g=\sum_{m=1}^\infty \sum_{j=0}^n  {n\choose j}(-\beta)^{n-j} (\beta-\lambda_m)^j\langle g, \nu_m \rangle \nu_m=\sum_{m=1}^\infty (-\lambda_m)^n\langle g, \nu_m \rangle \nu_m=\mathrm{A}_p^n g,
\end{align*}
for all $g\in D(\mathrm{A}_p).$

We now estimate $\|\mathscr{G}_{\alpha,\beta,p}^\nu(t)\|$ for any $t\geq 0.$  Since $\{\nu_m:m\in\mathbb{N}\}$ is an orthonormal basis, from definition of $\mathscr{G}_{\alpha,\beta,p}^\nu(t)$ it follows that
\begin{align*}
\|\mathscr{G}_{\alpha,\beta,p}^\nu(t)g\|^2= e^{-\beta t}\sum_{m=1}^\infty\sum_{k=0}^\infty\frac{(\beta-m^2)^kt^k}{\Gamma(k+1)}\sum_{j=0}^\infty \frac{(-\alpha t^{\nu+1} m^2)^j}{\Gamma(j+1)}\frac{\Gamma(k+j+1)}{\Gamma(j\nu+k+j+1)}|\langle g, \nu_m \rangle|^2,
\end{align*}
for $g\in \mathbb{W}_p.$
From \cite[Inequality (4.4)]{Bu-Is-86}, it follows that 
\begin{equation*}
    \frac{\Gamma(k+j+1)}{\Gamma(j\nu+k+j+1)}<\frac{1}{\Gamma(\nu j +1)},
\end{equation*}
and by \cite[Corollary 2]{Qi-02}, we obtain that
\begin{equation*}
\frac{1}{\Gamma(\nu j +1)}<\frac{1}{e^{\gamma\mu j}}<1,
\end{equation*}
where $\gamma=0.5772...$ is the Euler's constant. This implies that
\begin{align*}
\|\mathscr{G}_{\alpha,\beta,p}^\nu(t)g\|^2&\leq e^{-\beta t}\sum_{m=1}^\infty\sum_{k=0}^\infty\frac{(\beta-m^2)^kt^k}{\Gamma(k+1)}\sum_{j=0}^\infty \frac{(-\alpha t^{\nu+1} m^2)^j}{\Gamma(j+1)}|\langle g, \nu_m \rangle|^2\\
&=\sum_{m=1}^\infty e^{-m^2t(1+\alpha t^\nu)}|\langle g, \nu_m \rangle|^2.
\end{align*}
Since $\alpha>0,$ the Bessel inequality implies that $\|\mathscr{G}_{\alpha,\beta,p}^\nu(t)g\|\leq \|g\|.$ Therefore, $\|\mathscr{G}_{\alpha,\beta,p}^\nu(t)\|=\sup_{\|g\|\leq 1}\|\mathscr{G}_{\alpha,\beta,p}^\nu(t)g\|\leq 1$ for any $t\geq 0.$ 
\vskip 0.1 cm
\noindent\textbf{Step 2:} \emph{Functional setting and approximate controllability.} 
	Let us take $w(t)(\xi):=w(t,\xi)$ for $t\in J $ and $ \xi\in[0,\pi]$ and the operator $\mathrm{B}:\mathbb{U}\to\mathbb{W}_p$ is defined as  $$\mathrm{B}u(t)(\xi):=\eta(t,\xi)=\int_{0}^{\pi}G(\zeta,\xi)u(t)(\zeta)\mathrm{d}\zeta, \ t\in J,\ \xi\in [0,\pi],$$ with kernel $G\in\mathrm{C}([0,\pi]\times[0,\pi];\mathbb{R})$ and $G(\zeta,\xi)=G(\xi,\zeta),$ for all $\zeta,\xi\in [0,\pi]$.  
 We can choose the kernel $G(\cdot,\cdot)$ such that the operator $\mathrm{B}$ is one-one. In particular, 
 \begin{equation*}
		G(\zeta,\xi)=
		\begin{dcases}
			\xi(\pi-\zeta), \mbox{if}\ 0\le\zeta\le\xi\le\pi,\\
			(\pi-\xi)\zeta, \mbox{if}\ 0\le\xi\le\zeta\le\pi.
		\end{dcases}
	\end{equation*}
	We now prove that the operator $\mathrm{B}$ is bounded. For this, let us estimate 
	\begin{align*}
	\left\|\mathrm{B}u(t)\right\|_{\mathbb{W}_p}^p=\int_{0}^{\pi}\left|\int_{0}^{\pi}G(\zeta,\xi)u(t)(\zeta)\mathrm{d}\zeta\right|^p\mathrm{d}\xi.
	\end{align*}
	Using the Cauchy-Schwarz inequality, we obtain
\begin{align*}\left\|\mathrm{B}u(t)\right\|_{\mathbb{X}_p}^p&\le\int_{0}^{\pi}\left[\left(\int_{0}^{\pi}|G(\zeta,\xi)|^2\mathrm{d}\zeta\right)^{\frac{1}{2}}\left(\int_{0}^{\pi}|u(t)(\zeta)|^2\mathrm{d}\zeta\right)^{\frac{1}{2}}\right]^{p}\mathrm{d}\xi\\&=\left(\int_{0}^{\pi}|u(t)(\zeta)|^2\mathrm{d}\zeta\right)^{\frac{p}{2}}\int_{0}^{\pi}\left(\int_{0}^{\pi}|G(\zeta,\xi)|^2\mathrm{d}\zeta\right)^{\frac{p}{2}}\mathrm{d}\xi.
	\end{align*}
	Since the kernel $G(\cdot,\cdot)$ is continuous, we have
	\begin{align*}
	\left\|\mathrm{B}u(t)\right\|_{\mathbb{X}_p}\le C\left\|u(t)\right\|_{\mathbb{U}},
	\end{align*}
	so that we get 
	$\left\|\mathrm{B}\right\|_{\mathcal{L}(\mathbb{U};\mathbb{X}_p)}\le C.$ Thus, the operator $\mathrm{B}$ is bounded.

We choose
$$h(t,w(t,\xi))=k_{0}\cos\left(\frac{2\pi t}{T}\right)\sin(w(t,\xi)),t\in J, \xi\in[0,\pi],$$
where $k_{0}$ is some positive constant.
Next, we define $f:J\times \mathbb{W}_p\rightarrow \mathbb{W}_p$  as  
\begin{align*}
f(t,w(t))(\xi):= h(t,w(t,\xi))=k_{0}\cos\left(\frac{2\pi t}{T}\right)\sin(w(t)(\xi)), \ \xi\in[0,\pi].
\end{align*}
 Clearly, $f$ is continuous  and
\begin{align*}
\left\|f(t,w(t))\right\|_{\mathrm{L}^p}=k_0\left(\int_{0}^{\pi}\left|\cos\left(\frac{2\pi t}{T}\right)\sin(w(t)(\xi))\right|^p\mathrm{d}\xi\right)^{1/p}\le k_0 {\pi}^{\frac{1}{p}}\left|\cos\left(\frac{2\pi t}{T}\right)\right|.
\end{align*}
It is clear from the above facts that the function satisfy the conditions \textbf{\textit{(F1)}}-\textbf{\textit{(F2)}}.

Using the above expressions, one can reformulate the system \eqref{57} in an abstract form as given in \eqref{SEq} that satisfies all the assumptions. It remains to verify the approximate controllability of the corresponding linear system associated with \eqref{SEq}. To achieve this, we consider $\mathrm{B}^*\mathscr{G}_{\alpha,\beta,p}^\nu(T-t)^*w^{*}=0,$ for any $w^{*}\in\mathbb{W}^{*}_p$. Then, we have
	\begin{align*}
		\mathrm{B}^*\mathscr{G}_{\alpha,\beta,p}^\nu(T-t)^*w^{*}=0\Rightarrow \mathscr{G}_{\alpha,\beta,p}^\nu(T-t)^*w^{*}=0\Rightarrow w^{*}=0.
	\end{align*}
	Therefore, the approximate controllability of the linear system corresponding to  \eqref{SEq} follows directly from Lemma \ref{lem3.4}. Finally, by applying Theorem \ref{thm3.11}, we conclude that the semilinear system \eqref{SEq}, which is equivalent to the system \eqref{57} is approximately controllable.
    \vskip 0.1 cm
    \noindent\textbf{Case 2:} In this case we choose $\mathbb{W}=\mathbb{U}=\mathrm{L}^2([0,\pi];\mathbb{R})$ and the operator $\mathrm{B}=\mathrm{I}$. Then by the definition of the operator $F(\cdot)$ given in \eqref{op11}, we have
     \begin{align*}
        F(t)w^*&=\int^{t}_{0}\mathscr{G}_{\alpha,\beta}^\nu(t-s)\mathscr{G}_{\alpha,\beta}^\nu(T-s)^{*}w^*\mathrm{d}s.
    \end{align*}
     The operator $F(t)F(T)^{-1}:\mathcal{R}(F(T))\to\mathbb{W}$ is uniformly bounded for all $t\in J$ as the operators $\mathscr{G}_{\alpha,\beta}^\nu(\cdot)$ and $\mathscr{G}_{\alpha,\beta}^\nu(\cdot)^*$ are uniformly bounded. Hence, Assumption \textbf{\textit{(F)}} holds.

    In this case we take
$$h(t,w(t,\xi))=e^{-\mu t}w(t,\xi),t\in J, \xi\in[0,\pi], \mu>0,$$ 
and then define $f:J\times \mathbb{W}\rightarrow \mathbb{W}$  as  
\begin{align*}
f(t,w(t))(\xi):= h(t,w(t,\xi))=e^{-\mu t}w(t)(\xi), \ \xi\in[0,\pi].
\end{align*}
It is clear that $f$ is continuous and 
\begin{align*}
\left\|f(t,w(t))\right\|_{\mathbb{W}}=\left(\int_{0}^{\pi}\left|e^{-\mu t}w(t)(\xi)\right|^2\mathrm{d}\xi\right)^{1/2}\le e^{-\mu t}r\pi=\gamma_{r}(t), \ \mbox{for}\ \left\|w(t)\right\|_{\mathbb{W}}\le r.
\end{align*}
Therefore, the conditions  \textbf{\textit{(F1)}} and \textbf{\textit{(F3)}} are satisfied. 

Next, we verify the condition \eqref{cc4} stated in Theorem \ref{thm4.4}. In order to do this, we choose $r$ such that
\begin{align}\label{ce1}
    \|\zeta\|_{\mathbb{W}}+2\tilde{L}\left(\|\zeta_1\|_{\mathbb{W}}+\|\zeta\|_{\mathbb{W}}\right)<\frac{r}{2}.
\end{align}
Next, we estimates
\begin{align}\label{ce2}
    \|\gamma_{r}\|_{\mathrm{L}^1(J;\mathbb{R}^+)}+ 2\tilde{L}\|\gamma_{r}\|_{\mathrm{L}^1(J;\mathbb{R}^+)}&=\mu(1-e^{-\mu T})(1+2\tilde{L})r\pi.
\end{align}
From the above fact, it is ensure that we can choose a $\mu>0$ (sufficiently small) such that
\begin{align}\label{ce3}
    \mu\pi(1-e^{-\mu T})(1+2\tilde{L})\le \frac{1}{2}.
\end{align}
Combining the estimates \eqref{ce1}, \eqref{ce2}, and \eqref{ce3}, we conclude that the condition \eqref{cc4} is satisfied.

In the end, we verify the approximate controllability of the corresponding linear system associated with \eqref{SEq}. To achieve this, we consider $\mathrm{B}^*\mathscr{G}_{\alpha,\beta}^\nu(T-t)^*w^{*}=0,$ for any $w^{*}\in\mathbb{W}$. Then, we have
	\begin{align*}
		\mathrm{B}^*\mathscr{G}_{\alpha,\beta}^\nu(T-t)^*w^{*}=0\Rightarrow \mathscr{G}_{\alpha,\beta}^\nu(T-t)^*w^{*}=0\Rightarrow w^{*}=0.
	\end{align*}


Therefore, the linear system corresponding to \eqref{SEq} is approximately controllable. Finally, applying Theorem \ref{thm4.4}, we deduce that the semilinear system \eqref{SEq}, equivalent to the system \eqref{57} is achieved the approximate controllability.

The following Proposition investigates the approximate controllability of the abstract linear system corresponding to \eqref{57} is approximately controllable on $[0,t]$ for all $0<t\le T$ and corresponds to an extension of the rank condition for $C_0$-semigroups (\cite[Theorem 3.16]{Cu-Pr}).

\begin{pro}\label{Prop4.2}
 The abstract linear system corresponding to \eqref{57} is approximately controllable on $[0,t]$ for all $0<t\le T$, if Span$\{\mathrm{A}^n\mathrm{B}U_{\infty}: n\in\mathbb{N}_0\}$ is dense in $\mathbb{W}_p$ for $p\in[2,\infty)$, where $U_\infty=\{u\in\mathbb{U}: \mathrm{B}u\in\cap_{k=1}^{\infty} D(\mathrm{A}^n)\}$.
\end{pro}

\section{Concluding Remarks}
This study was began by proving the well-posedness of the system of the linear non-homogeneous integrodifferential equation give in \eqref{NHS} by showing the existence of the resolvent family $\mathscr{G}_{\alpha,\beta}^\nu(\cdot)$. Further, we developed some important properties associated with the resolvent family. Consequentely, using the resolvent family and  the Schauder fixed point theorem, we obtained the existence of a mild solution of the system \eqref{SEq} for a given control $u\in\mathrm{L}^{2}(J;\mathbb{U})$. Next, we examined the approximate controllability issue for the linear system \eqref{LEq}, focusing on the associated optimization problem considered in Theorem \ref{exi} and determining the expression for the optimal control, as presented in Lemma \ref{lem3.2}. Further, we derived sufficient conditions of the approximate controllability for the semilinear system \eqref{SEq} in a super-reflexive Banach space (see Theorems \ref{thm3.11} and \ref{thm4.4}). We also discussed the approximate controllability of the semilinear system \eqref{SEq} in a general Banach spaces. In the end, we utilized our results to study the approximate controllability of a heat equation with singular memory.

\begin{appendix}
		\renewcommand{\thesection}{\Alph{section}}
		\numberwithin{equation}{section}
	
	\section{}\label{ap}\setcounter{equation}{0}
In this section, we collect auxiliary results that are useful in the developments of the paper.
    \begin{lem}\label{lem2.12}
	If the resolvent operator $\mathscr{G}_{\alpha,\beta}^\nu(t)$ is compact for $t>0$, then the operator $\mathcal{I}:\mathrm{L}^{2}(J;\mathbb{W})\rightarrow C(J;\mathbb{W})$ is given by
	\begin{align*}
		(\mathcal{I}g)(t)= \int^{t}_{a}\mathscr{G}_{\alpha,\beta}^\nu(t-s)g(s)\mathrm{d}s, \ t\in J,
	\end{align*}
	is compact.
\end{lem}
\begin{lem}\label{lem5.3.2}Let the operator $\mathcal{G}:\mathrm{L}^1(J;\mathbb{W})\to C([a,b];\mathbb{W})$ as 
	\begin{align}\label{eqn5.3.2}
		(\mathcal{G}g)(t):=\int_{a}^{t}\mathscr{G}_{\alpha,\beta}^\nu(t-s)g(s)\mathrm{d}s,\ t\in J.
	\end{align}
	If $\{g_n\}_{n=1}^{\infty}\subset\mathrm{L}^1(J;\mathbb{W})$ is any integrably bounded sequence, then the sequence $v_n:=\mathcal{G}(g_n)$ is relatively compact.
\end{lem}
\begin{cor}\label{cor1}
	Let $\{g_n\}_{n=1}^{\infty}\subset\mathrm{L}^1(J;\mathbb{W})$ be an integrably bounded sequence such that $$g_n\xrightharpoonup{w} g \ \text{ in }\ \mathrm{L}^1(J;\mathbb{W}) \ \mbox{as}\ n\to\infty.$$ Then
	$$\mathcal{G}(g_n)\to\mathcal{G}(g)\ \text{ in }\ C(J;\mathbb{W}) \ \mbox{as}\ n\to\infty.$$ 
\end{cor}
Now we introduce some special class of Banach spaces and reviews their characterization.

\begin{Def} A Banach space $\mathbb{W}$ is called \emph{strictly convex} if for any $w_1,w_2\in\mathbb{W}$ with $\left\|w_1\right\|_{\mathbb{W}}=\left\|w_2\right\|_{\mathbb{W}}=1$, the inequality
	$$\left\|\rho z_1+(1-\rho)z_2\right\|_{\mathbb{W}}<1,\ \mbox{for}\ 0<\rho<1,$$ holds.
\end{Def}
\begin{Def} A Banach space $\mathbb{W}$ is called \emph{uniformly convex} if for any given $\epsilon>0,$ and for all $w_1,w_2\in\mathbb{W}$ with $\|w_1\|_{\mathbb{W}}=1$ and $\|w_2\|_{\mathbb{W}}=1$, there exists a $\delta=\delta(\epsilon)$ such that
	$$\left\|w_1-w_2\right\|_{\mathbb{W}}\ge\epsilon \implies \left\|w_1+w_2\right\|_{\mathbb{W}}\le 2(1-\delta).$$
\end{Def}
\begin{Def}
	A Banach space $\mathbb{W}$ is said to be \emph{uniformly smooth} if for any given $\epsilon>0$, there exists a $\delta>0$ such that  $$\frac{1}{2}\left(\|w_1+w_2\|_{\mathbb{W}}+\|w_1-w_2\|_{\mathbb{W}}\right)-1\le\epsilon\|w_2\|_{\mathbb{W}},$$ where  $w_1,w_2\in\mathbb{W}$ with $\|w_1\|_{\mathbb{W}}=1$ and $\|w_2\|_{\mathbb{W}}\le\delta$.
\end{Def}
Now we recall the concept of the \emph{modulus of convexity} and the \emph{modulus of smoothness}  of a Banach space $\mathbb{W}$.

The \emph{modulus of convexity} of the Banach space $\mathbb{W}$ is given as
\begin{align*}
	\delta_{\mathbb{W}}(\epsilon)=\inf\left\{1-\frac{\|w_1+w_2\|_{\mathbb{W}}}{2}:\|w_1\|_{\mathbb{W}}\le 1, \|w_2\|_{\mathbb{W}}\le 1, \|w_1-w_2\|_{\mathbb{W}}\ge\epsilon\right\}.
\end{align*}
Note that the function $\delta_{\mathbb{W}}(\cdot)$ is defined on the interval $[0,2]$  which is non-decreasing. Moreover, a Banach space $\mathbb{W}$ is \emph{uniformly convex} if and only if $\delta_{\mathbb{W}}(\epsilon)>0$ for all $\epsilon>0$ (cf. \cite{GP1975}).

Next, the function defined by 
\begin{align*}
	\rho_{_{\mathbb{W}}}(\tau)&=\sup\left\{\frac{\|w_1+\tau w_2\|_{\mathbb{W}}}{2}+\frac{\|w_1-\tau w_2\|_{\mathbb{W}}}{2}-1 : \|w_1\|_{\mathbb{W}}=\|w_2\|_{\mathbb{W}}=1\right\},
    \end{align*}
for all $\tau\in[0, \infty)$, is called the \emph{modulus of smoothness} of the space $\mathbb{W}$.

\noindent 	A Banach space $\mathbb{W}$  is uniformly smooth if and only if (cf. \cite{GP1975}) $\lim_{t\to 0^+}\frac{\rho_{_{\mathbb{W}}}(\tau)}{\tau}=0.$

\begin{rem}\label{r2}\cite{CC2009,GP1975}.
	\begin{itemize}
    \item [(a)] Every \emph{uniformly convex} space $\mathbb{W}$ is \emph{strictly convex}. 
		\item[(b)] A Banach space $\mathbb{W}$ is called \emph{$p$-uniformly smooth} ($1<p\le2$), if there exist an equivalent norm on $\mathbb{W}$ holds the property $\rho_{_{\mathbb{W}}}(\tau)\le C\tau^p$ for some constant $C$. Similarly, a Banach space is said to be \emph{$q$-uniformly convex} ($2\le q< \infty$), if there exist an equivalent norm on $\mathbb{W}$ satisfies $\delta_{\mathbb{W}}(\epsilon)\ge C\epsilon^q$ for some constant $C>0$.
        \item[(c)] If $\mathbb{W}^*$ is $p$-smooth, then $\mathbb{W}$ is $q$-convex with $\frac{1}{p}+\frac{1}{q}=1$.
		\item[(d)] A Banach space $\mathbb{W}$ is uniformly smooth if and only if its dual space $\mathbb{W}^*$ is uniformly convex, and vice versa. Additionally, any uniformly convex or uniformly smooth space is $q$-uniformly convex and $p$-uniformly smooth for some $q<\infty$ and $p>1$.
        \item [(d)] Every uniformly convex or uniformly smooth space is super-reflexive.
	\end{itemize}
\end{rem}
\begin{Ex}
	Here are some examples of a super-reflexive space. Let $\Omega\subset\mathbb{R}^n$ be a measurable set.
	\begin{itemize}
		\item [(i)]  The space $\mathrm{L}^p(\Omega)$ for $2\le p<\infty$.
		\item [(ii)] The Sobolev space $\mathrm{W}^{m,p}(\Omega)$ for $2\le p<\infty$ and $m\in\mathbb{N}$.
	\end{itemize}
 \end{Ex}
    \end{appendix}

\medskip\noindent	{\bf Declarations} 

\par {\bf Conflict of interest:} The author has not disclosed any conflict of interest.

\end{document}